\documentclass[a4paper]{article}

\usepackage{color}
\usepackage{graphicx}
\usepackage{stmaryrd}
\usepackage[all,2cell]{xy}
\xyoption{rotate}
\UseAllTwocells

\usepackage{amssymb,amsmath}
\usepackage{graphicx}
\newcommand{\hyph}{\text{-}}
\newcommand{\bref}[1]{(\ref{#1})}
\newcommand{\cat}[1]{\scat{#1}}
\newcommand{\fcat}[1]{\mathbf{#1}}
\newcommand{\ovln}[1]{\overline{#1}}
\newcommand{\such}{:}
\newcommand{\without}{\setminus}
\newcommand{\epsln}{\varepsilon}
\newcommand{\Hom}{\mathrm{Hom}}
\newcommand{\op}{\mathrm{op}}
\newcommand{\id}{\mathrm{id}}
\newcommand{\Ab}{\fcat{Ab}}

\newcommand{\Set}{\fcat{Set}}
\newcommand{\CAT}{\fcat{CAT}}
\newcommand{\oppairu}{\rightleftarrows}
\newcommand{\Mod}{\fcat{Mod}}
\newcommand{\swnt}{\rotatebox{45}{$\Leftarrow$}\!}
\newcommand{\demph}[1]{\textbf{\textup{#1}}}
\newcommand{\scat}[1]{\mathcal{#1}}
\newcommand{\iso}{\cong}
\newcommand{\eqv}{\simeq}
\newcommand{\of}{\circ}
\newcommand{\sub}{\subseteq}
\newcommand{\ladj}{\dashv}
\newcommand{\cell}[4]{\put(#1,#2){\makebox(0,0)[#3]{#4}}}
\newcommand{\Zero}{\mathbf{0}}
\newcommand{\One}{\mathbf{1}}
\newcommand{\Two}{\mathbf{2}}
\newcommand{\toby}[1]{\stackrel{#1}{\longrightarrow}}
\newcommand{\Top}{\fcat{Top}}
\newcommand{\Vect}{\fcat{Vect}}
\newcommand{\incl}{\hookrightarrow}
\DeclareMathOperator{\im}{im}
\newcommand{\from}{\colon}
\newcommand{\jn}{\vee}
\newcommand{\Q}{\mathbb{Q}}
\renewcommand{\emptyset}{\varnothing}
\renewcommand{\implies}{\mathrel{\Rightarrow}}
\newcommand{\hardref}[1]{#1} 
\newcommand{\pfun}{\mathrel{\ooalign{\hfil$+\mkern3mu$\hfil\cr$\longrightarrow$\cr}}} 
\DeclareMathOperator{\dn}{\downarrow} 
\DeclareMathOperator{\up}{\uparrow} 

\newcommand{\truncsub}
{\ensuremath{\mathbin{\hbox{%
        \ooalign{\hfil\raisebox{3pt}{$\cdot$}%
          \hfil\cr\hfil$-$\hfil}}}}}

\newcommand{\Inv}{\fcat{Inv}}
\newcommand{\Env}{\fcat{Env}}
\newcommand{\ADJ}{\fcat{ADJ}}
\newcommand{\ADJi}{\ADJ_{\mathord{\cong}}}

\newcommand{\refl}{\cat{R}}
\newcommand{\IEnv}{\cat{I}}
\DeclareMathOperator{\colim}{colim}
\DeclareMathOperator{\Cone}{Cone}

\usepackage{scalerel,stackengine}
\stackMath
\newcommand{\snughat}{\setlength{\unitlength}{1ex}%
\begin{picture}(1.25,0.42)%
\put(-0.25,-1.33){$\widehat{\ }$}
\end{picture}}
\newlength{\wclen}

\newcommand\widercheck[1]{%
\setlength{\wclen}{\widthof{\ensuremath{#1}}}%
\savestack{\tmpbox}{\stretchto{%
  \scaleto{%
    \scalerel*[.5\wclen]{\kern-.6pt\snughat\kern-.6pt}%
    {\rule[-\textheight/2]{1ex}{\textheight}}
  }{\textheight}%
}{.45ex}}%
\stackon[1pt]{#1}{\scalebox{-1}{\tmpbox}}%
}

\newcommand{\cocmp}[1]{\widehat{#1}}
\newcommand{\cmp}[1]{\widercheck{#1}}
\newcommand{\du}[1]{{#1}^\vee}
\newcommand{\dubk}{\du{}}
\newcommand{\ddu}[1]{{#1}^{\vee\vee}}
\newcommand{\dddu}[1]{{#1}^{\vee\vee\vee}}
\newcommand{\nv}[1]{N_{#1}}
\newcommand{\conv}[1]{N^{#1}}
\newcommand{\rnv}[1]{N(#1)}
\newcommand{\rrnv}[1]{NN(#1)}
\newcommand{\pof}{\odot}
\newcommand{\etimes}{\boxtimes}
\newcommand{\V}{\cat{V}}
\newcommand{\As}{\scat{A}}
\newcommand{\A}{\cat{A}}
\newcommand{\Bs}{\scat{B}}
\newcommand{\B}{\cat{B}}
\newcommand{\Cs}{\scat{C}}
\newcommand{\C}{\cat{C}}
\newcommand{\lhom}[1]{[#1]_{\mathord{<}}}
\newcommand{\rhom}[1]{[#1]_{\mathord{>}}}
\newcommand{\Prof}{\fcat{Prof}}
\DeclareMathOperator{\supp}{supp}
\newcommand{\ltoby}[1]{\xrightarrow{#1}}
\newcommand{\VNat}{\V\hyph\textbf{\textup{Nat}}}

\usepackage[amsmath,thmmarks]{ntheorem}

\newtheorem{thm}{Theorem}[section]
\newtheorem{propn}[thm]{Proposition}
\newtheorem{lemma}[thm]{Lemma}
\newtheorem{cor}[thm]{Corollary}

\theorembodyfont{\normalfont}

\newtheorem{defn}[thm]{Definition}
\newtheorem{example}[thm]{Example}

\newtheorem{remark}[thm]{Remark}

\theoremstyle{nonumberplain}
\theoremsymbol{\ensuremath{\square}}
\qedsymbol{\ensuremath{\square}}

\newtheorem{proof}{Proof}

\usepackage[numbers]{natbib}
\usepackage{natbibspacing}
\renewcommand{\bibfont}{\small}

\definecolor{myurlcolor}{rgb}{0.1,0.1,0.8}
\definecolor{mylinkcolor}{rgb}{0.05,0.05,0.4}
\usepackage{hyperref}
\hypersetup{colorlinks,
linkcolor=mylinkcolor,
citecolor=mylinkcolor,
urlcolor=mylinkcolor}

\usepackage{tocloft}
\title{Isbell conjugacy and the reflexive completion}
\author{Tom Avery \qquad Tom Leinster%
\thanks{School of Mathematics, University of Edinburgh, Edinburgh EH9 3FD,
  Scotland; Tom.Leinster@ed.ac.uk. Supported by a Leverhulme Trust Research
  Fellowship.}}
\date{}

\begin{document}
\sloppy

\maketitle

\begin{abstract}
The reflexive completion of a category consists of the $\Set$-valued
functors on it that are canonically isomorphic to their double
conjugate. After reviewing both this construction and Isbell conjugacy
itself, we give new examples and revisit Isbell's main results from 1960
in a modern categorical context. We establish the sense in which reflexive
completion is functorial, and find conditions under which two categories
have equivalent reflexive completions. We describe the relationship between
the reflexive and Cauchy completions, determine exactly which limits and
colimits exist in an arbitrary reflexive completion, and make precise the
sense in which the reflexive completion of a category is the intersection
of the categories of covariant and contravariant functors on it.
\end{abstract}

\tableofcontents

\section{Introduction}
\label{sec:intro}

Isbell conjugacy inhabits the same basic level of category theory
as the Yoneda lemma, springing from the most primitive concepts of the
subject: category, functor and natural transformation. It can be understood
as follows.

Let $\As$ be a small category. Any functor $X \from \A^\op \to \Set$ gives
rise to a new functor $X' \from \A^\op \to \Set$ defined by
\[
X'(a) = [\A^\op, \Set](\A(-, a), X),
\]
and so, in principle, an infinite sequence $X, X', X'', \ldots$ of functors
$\A^\op \to \Set$. Of course, they are all canonically isomorphic, by the
Yoneda lemma. But $X$ also gives rise to a functor $\du{X} \from \A
\to \Set$, its \emph{Isbell conjugate}, defined by
\begin{align}
\label{eq:conj-intro}
\du{X}(a) = [\A^\op, \Set](X, \A(-, a)).
\end{align}
The same construction with $\A$ in place of $\A^\op$ produces from $\du{X}$
a further functor $\ddu{X} \from \A^\op \to \Set$, and so on, giving an
infinite sequence $X, \du{X}, \ddu{X}, \ldots$ of functors on $\A$ with
alternating variances. Although it makes no sense to ask whether $\du{X}$
is isomorphic to $X$ (their types being different), one can ask
whether $\ddu{X} \iso X$. This is false in general. Thus, there is nontrivial
structure.

The conjugacy operations define an adjunction between
$[\A^\op, \Set]$ and $[\A, \Set]^\op$, so that
\[
[\A^\op, \Set](X, \du{Y}) \iso [\A, \Set](Y, \du{X})
\]
naturally in $X \from \A^\op \to \Set$ and $Y \from \A \to \Set$.  The unit
and counit of the adjunction are canonical maps $X \to \ddu{X}$ and $Y \to
\ddu{Y}$, and a covariant or contravariant functor on $\A$ is said to be
\emph{reflexive} if the canonical map to its double conjugate is an
isomorphism.

The \emph{reflexive completion} $\refl(\A)$ of $\A$ is the category of
reflexive functors on $\A$ (covariant or contravariant; it makes no
difference). Put another way, $\refl(\A)$ is the invariant part of the
conjugacy adjunction. It contains $\A$, since representables are
reflexive. Its properties are the main subject of this work.

The reflexive completion is very natural category-theoretically, but
categories of reflexive objects also appear in other parts of
mathematics. That is, there are many notions of duality in mathematics, in
most instances there is a canonical map $\eta_X \from X \to X^{**}$ from
each object $X$ to its double dual, and special attention is paid to those
$X$ for which $\eta_X$ is an isomorphism. For example, in linear algebra,
the vector spaces $X$ with this property are the finite-dimensional ones,
and in functional analysis, there is a highly developed theory of
reflexivity for Banach spaces and topological vector spaces.

\paragraph*{Content of the paper} We begin with the definition of
conjugacy on \emph{small} categories, giving several characterizations of
the conjugacy operations and many examples (Sections~\ref{sec:conj-small}
and~\ref{sec:eg-conj}). Defining conjugacy on an arbitrary category is
more delicate, and we review and use the notion of small functor
(Section~\ref{sec:conj-gen}). This allows us to state the definition of the
reflexive completion of an arbitrary category, and again, we give many
examples (Sections~\ref{sec:completion} and~\ref{sec:eg-compl}).

Up to here, there are no substantial theorems, but the examples provide
some surprises. For instance, the reflexive completion of a nontrivial
group is simply the group with initial and terminal objects
adjoined---except when the group is of order $2$, in which case it is
something more complicated (for reasons related to the fact that $2 + 2 = 2
\times 2$; see Example~\ref{eg:compl-groups}). There is also a finite
monoid whose reflexive completion is not even small, a fact due to Isbell
(Examples~\ref{eg:compl-7} and~\ref{eg:char-7}). Other examples involve
the Dedekind--MacNeille completion of an ordered set
(Examples~\ref{eg:compl-posets} and~\ref{eg:compl-posets-2}) and the tight
span of a metric space (at the end of Section~\ref{sec:eg-compl}).

The second half of the paper develops the theory, as follows.

Section~\ref{sec:density} collects necessary results on dense and adequate
functors. (See Definition~\ref{defn:adeq} and Remark~\ref{rmk:dense-adeq}
for this terminology.) Many of them are standard, but
we address points about set-theoretic size that do not seem to have
previously been considered. Using the results of Section~\ref{sec:density},
we give a unique characterization of the reflexive completion that sharpens
a result of Isbell's (Theorem~\ref{thm:compl-char}).

Reflexive completion is functorial (Section~\ref{sec:func}), but only with
respect to a very limited class of functors: the small-adequate
ones. It is often the case that the functor $\refl(F) \from
\refl(\B) \to \refl(\A)$ induced by a functor $F \from \A \to \B$ is an
equivalence. For example, using work of Day and Lack on small functors
together with the results on size just mentioned, we show that $\refl(F)$
is always an equivalence if $\B$ is either small or both complete and
cocomplete (Corollary~\ref{cor:gentle-func}). 

Our study of functoriality naturally recovers Isbell's result that
reflexive completion is idempotent: $\refl(\refl(\A)) \eqv \refl(\A)$. A
category is \emph{reflexively complete} if it is the reflexive completion
of some category, or equivalently if every reflexive functor
on it is representable.

Reflexive completion has certain formal resemblances to Cauchy completion,
but the reflexive completion is typically bigger
(Figure~\ref{fig:completions}). The relationship is analysed in
Section~\ref{sec:cauchy}. 

A reflexively complete category has absolute (co)limits, and if it is the
reflexive completion of a \emph{small} category then it has initial and
terminal objects too, but these are all the limits and colimits that it
generally has (Section~\ref{sec:limits}). The case of ordered sets, where
the reflexive (Dedekind--MacNeille) completion has all (co)limits, is
atypical. On the other hand, it \emph{is} true that a complete or
cocomplete category is reflexively complete (Figure~\ref{fig:completions}).

\begin{figure}
\setlength{\unitlength}{1mm}
\begin{picture}(120,43)
\cell{2}{5}{bl}{%
$\xymatrix@R=2ex@C=2ex{
        &\IEnv(\A)      &       \\
\cocmp{\A} \ar@{-}[ru] \ar@{-}[rd]      &
        &
\cmp{\A} \ar@{-}[lu] \ar@{-}[ld]        \\
        &\refl(\A)      &       \\
        &\ovln{\A} \ar@{-}[u] \ar@{-}[d]        &       \\
        &\A
}$%
}
\cell{118}{5}{br}{\includegraphics[height=37\unitlength]{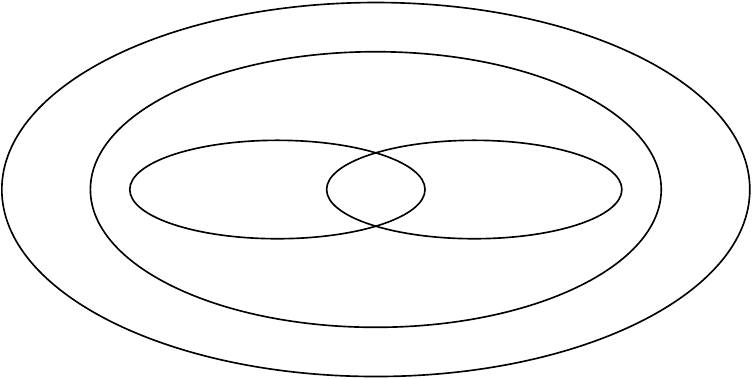}}
\cell{70}{26}{b}{\colorbox{white}{complete}}
\cell{93.5}{26}{b}{\colorbox{white}{cocomplete}}
\cell{81}{33}{b}{\colorbox{white}{reflexively complete}}
\cell{81}{39}{b}{\colorbox{white}{Cauchy complete}}
\cell{16}{-0.5}{b}{(a)}
\cell{81}{-0.5}{b}{(b)}
\end{picture}%
\caption{(a)~Completions of a category $\A$: the Cauchy completion
$\ovln{\A}$, reflexive completion $\refl(\A)$, free completions
$\cocmp{\A}$ and $\protect\cmp{\A}$ with respect to small colimits and
small limits, and Isbell envelope $\IEnv(\A)$; (b)~classes of complete
categories.}
\label{fig:completions}
\end{figure}

Informally, one can understand $\refl(\A)$ as the intersection $\cocmp{\A} \cap
\cmp{\A}$, where $\cocmp{\A}$ and $\cmp{\A}$ are the free completions of
$\A$ under small colimits and small limits. (If $\A$ is small then
$\cocmp{\A} = [\A^\op, \Set]$ and $\cmp{\A} = [\A, \Set]^\op$.)
Section~\ref{sec:envelope} formalizes this idea, reviewing the definition
of the Isbell envelope $\IEnv(\A)$ of a category and proving that the
square in Figure~\ref{fig:completions}\hardref{(a)} is a pullback in the
bicategorical sense.

We work with categories enriched over a suitable monoidal category $\V$ in
Sections~\ref{sec:conj-small}--\ref{sec:eg-compl}, then restrict to $\V =
\Set$ from Section~\ref{sec:density}. While some of the later results are
particular to $\V = \Set$ (such as Theorem~\ref{thm:rcn-lim} on limits),
others can be generalized to any $\V$.  To avoid complicating the
presentation, we have not specified exactly which results generalize, but
we have tried to choose proofs that make any generalization transparent.

\paragraph*{Relationship to Isbell's paper} Although there are many new
results in this work, some parts are accounts of results first proved
in Isbell's remarkable paper~\cite{IsbeAS}, and the reader may ask what we
bring that Isbell did not. There are several answers. 

First, Isbell's paper was extraordinarily early. He submitted it in
mid-1959, only the year after the publication of Kan's paper introducing
adjoint functors. What we now know about category theory can be used to
give shape to Isbell's original arguments. In Grothendieck's
metaphor~\cite{McLaRSG}, the rising sea of general category theory has made
the hammer and chisel unnecessary.

Second, Isbell worked only with full subcategories, where we use arbitrary
functors. It is true that we will often need to assume our functors to be
full and faithful, so that up to equivalence, they are indeed
inclusions of full subcategories. Nevertheless, the functor-based approach
has the benefits of being equivalence-invariant and of revealing exactly
where the full and faithful hypothesis is needed.  Ulmer emphasized that
many naturally occurring dense functors are not full and faithful
(Example~\ref{eg:adjt-dense}), and the theory of dense and adequate functors
should be developed as far as possible without that assumption.

Third, we modernize some aspects, including the treatment of set-theoretic
size. Isbell used a size constraint on $\Set$-valued functors that he
called properness and Freyd later called pettiness
(Remark~\ref{rmk:proper}). It now seems clear that the most natural such
notion is that of small functor, which extends smoothly to the
enriched context and is what we use here.

Finally, Isbell simply omitted several proofs; we provide them.

\paragraph*{Terminology} Isbell conjugacy has sometimes been called Isbell
duality 
(as in Di~Liberti~\cite{DiLi}), but that term has also been used for a
different purpose entirely (as in Barr, Kennison and
Raphael~\cite{BKR}). Reflexive completion has also been studied under the
name of Isbell completion (as in Willerton~\cite{WillTSI}).

\paragraph*{Conventions} 
Usually, and always in declarations such as `let $\A$ be a category', the
word `category' means \emph{locally small} category. However, we will
sometimes form categories such as $[\A^\op, \Set]$ that are not locally
small. For us, the words \demph{small} and \demph{large} refer to sets and
proper classes. A category is (co)complete when it admits small (co)limits.

The symbol $\times$ denotes both product and copower, so that when $S$ is a
set and $a$ is an object of some category, $S \times a = \coprod_{s \in S}
a$.

\section{Conjugacy for small categories}
\label{sec:conj-small}

Certain aspects of conjugacy are simpler for small categories. In this
section, we review several descriptions and characterizations of
conjugacy on small categories, all previously known. Our categories will be
enriched in a complete and cocomplete symmetric monoidal closed category
$\cat{V}$. Henceforth, we will usually abbreviate `$\V$-category' to
`category', and similarly for functors, adjunctions, etc.; all are
understood to be $\V$-enriched.

Let $\scat{A}$ be a small category. The \demph{(Isbell) conjugate} of a
functor $X \from \scat{A}^\op \to \cat{V}$ is the functor $\du{X}
\from \scat{A} \to \cat{V}$ defined by
\[
\du{X}(a) = [\scat{A}^\op, \cat{V}](X, \cat{A}(-, a))
\]
($a \in \scat{A}$). With $\scat{A}^\op$ in place of $\scat{A}$, this means
that the conjugate of a functor $Y \from \scat{A} \to \cat{V}$ is the
functor $\du{Y} \from \scat{A}^\op \to \cat{V}$ defined by
\[
\du{Y}(a) = [\scat{A}, \cat{V}](Y, \cat{A}(a, -)).
\]
Conjugacy defines a pair of functors
\begin{align}
\label{eq:conj-ftrs}
\xymatrix{
[\scat{A}^\op, \V] \ar@<.5ex>[r]        &
[\scat{A}, \V]^\op\ar@<.5ex>[l]         
}.
\end{align}

\begin{remark}
\label{rmk:unamb}
Writing $\vee_\scat{A}$ for the conjugacy functor $[\scat{A}^\op, \V] \to
[\scat{A}, \V]^\op$, the conjugacy functor in the opposite
direction is $\vee_{\scat{A}^\op}^\op$. 

There is a case for adopting different symbols for the two directions.
Others have done this; for example, Wood uses $X^+$ and $Y^-$
(\cite{WoodSRT}, Section~1). We could write $X^\vee = \vee_{\scat{A}}(X)$ and
$Y^\wedge = \vee_{\scat{A}^\op}(Y)$, so that $X^\vee \in \cmp{\scat{A}}$
and $Y^\wedge \in \cocmp{\scat{A}}$ (in the notation defined in
Section~\ref{sec:conj-gen}). But we will use the same symbol for both. This
is partly to emphasize that every functor $Z$ from a small category to $\V$
has a single, unambiguous, conjugate $\du{Z}$, which is the same whether
$Z$ is regarded as a covariant functor on its domain or a contravariant
functor on the opposite of its domain.
\end{remark}

Given $X\from \As^\op \to \V$ and $Y \from \As \to \V$, define 
\begin{align}
\label{eq:external}
\begin{array}{cccc}
X \etimes Y \from       &\As^\op \otimes \As    &\to            &\V     \\
                        &(a, b)                 &\mapsto        &
X(a) \otimes Y(b).
\end{array}
\end{align}
One verifies that 
\begin{align}
\label{eq:etimes-adjn}
[\As^\op, \V](X, \du{Y}) 
\iso 
[\As^\op \otimes \As, \V](X \etimes Y, \Hom_{\As})
\iso
[\As, \V](Y, \du{X})
\end{align}
naturally in $X$ and $Y$. In particular, the conjugacy
functors~\eqref{eq:conj-ftrs} are adjoint.

Evidently
\[
\du{\As(-, a)} \iso \As(a, -),
\qquad
\du{\As(a, -)} \iso \As(-, a)
\]
naturally in $a \in \As$. Thus, both triangles in the diagram
\begin{align}
\label{eq:conj-triangles}
\begin{array}{c}
\xymatrix@R-3ex{
{[\As^\op, \V]}\ar@/^/[rr]^-{\vee} & &
{[\As, \V]^\op}\ar@/^/[ll]^{\vee^\op}_{\perp} \\ 
	      &  & \\
	      & {\As}\ar[uul]^{H_{\bullet}}\ar[uur]_{H^{\bullet}} &
}
\end{array}
\end{align}
commute, where $H_{\bullet}$ and $H^{\bullet}$ are the two Yoneda
embeddings. This property characterizes conjugacy:

\begin{lemma}
\label{lemma:conj-unique}
Isbell conjugacy is the unique adjunction such that both triangles in 
\eqref{eq:conj-triangles} commute up to isomorphism.
\end{lemma}

\begin{proof}
Let $P \from [\As^\op, \V] \to [\As, \V]^\op$ be a left adjoint satisfying
$P \of H_{\bullet} \iso H^{\bullet}$. By hypothesis, $P(X) \iso \du{X}$
when $X$ is representable. But every object of $[\As^\op, \V]$ is a small
colimit of representables (as $\As$ is small), and both $P$ and $\du{(\ )}$
preserve colimits (being left adjoints), so $P \iso \du{(\ )}$. 
\end{proof}

Conjugacy can also be described as a nerve-realization adjunction. Any
functor $F \from \As \to \cat{E}$ induces a \demph{nerve functor}
\[
\begin{array}{cccc}
\nv{F} \from    &\cat{E}        &\to            &[\As^\op, \V]  \\
                &E              &\mapsto        &\As(F-, E).
\end{array}
\]
When $\cat{E}$ is cocomplete, the nerve functor has a left adjoint,
sometimes called the realization functor of $F$ (after the case where $F$
is the standard embedding of the simplex category $\Delta$ into $\Top$). It
is the left Kan extension of $F$ along the Yoneda embedding $H_\bullet$.

Taking $F$ to be $H^\bullet \from \As \to [\As, \V]^\op$, we thus obtain a
pair of adjoint functors between $[\As^\op, \V]$ and $[\As, \V]^\op$. This
is the conjugacy adjunction. For example, in
diagram~\eqref{eq:conj-triangles}, the functor $\du{(\ )} \from [\As^\op,
\V] \to [\As, \V]^\op$ is the left Kan extension of $H^\bullet$ along
$H_\bullet$.

Yet another derivation of conjugacy uses profunctors. Our convention is
that for small categories $\As$ and $\Bs$, a \demph{profunctor} $\Bs
\pfun \As$ is a functor $\As^\op \otimes \Bs \to \V$, and the composite of
profunctors $Q \from \Cs \pfun \Bs$ and $P \from \Bs \pfun \As$ is denoted by
$P \pof Q \from \Cs \pfun \As$. 

The operation of composition with a profunctor, on either the left or the
right, has a right adjoint. Indeed, given profunctors
\[
\xymatrix@C+.9em{
\Cs
\ar@/_1pc/[rr]|-@{|}_R
\ar[r]|-@{|}^Q
&
\Bs
\ar[r]|-@{|}^P
&
\As,
}
\]
there are profunctors
\[
\xymatrix@C+1em{
\Cs
\ar[r]|-@{|}^{\rhom{P, R}}
&
\Bs
\ar[r]|-@{|}^{\lhom{Q, R}}
&
\As
}
\]
defined by
\begin{align*}
\rhom{P, R}(b, c)       &
=
[\As^\op, \V](P(-, b), R(-, c)),        \\
\lhom{Q, R}(a, b)       &
=
[\Cs, \V](Q(b, -), R(a, -)),        
\end{align*}
which satisfy the adjoint correspondences
\begin{equation}
\label{eq:prof-adj}
\begin{tabular}{r@{\:}c@{\:}l}
$P$ &$\to$ &$\lhom{Q, R}$     
  \\[.3ex]
\hline  
\vphantom{$l^{l^l}$}
$P \pof Q$& $\to$ &$R$
\\[.3ex]
\hline
\vphantom{$l^{l^l}$}
$Q$ & $\to$ & $\rhom{P, R}$
\end{tabular}
\end{equation}
Now take $\Bs$ to be the unit $\V$-category $\scat{I}$, with $\Cs = \As$
and $R = \Hom_{\As}$. The profunctors $P$ and $Q$ are functors $X\from
\As^\op \to \V$ and $Y \from \As \to \V$, respectively. Then $\rhom{P, R} =
\du{X}$ and $\lhom{Q, R} = \du{Y}$, while $P \pof Q= X \etimes Y$, and the
general adjointness relations~\eqref{eq:prof-adj} reduce to the conjugacy
relations~\eqref{eq:etimes-adjn}.

Finally, conjugates can be described as Kan extensions or lifts in the
bicategory $\V\hyph\Prof$ of $\V$-profunctors.  Let $X \from \As^\op \to
\V$. There is a canonical natural transformation
\begin{align}
\label{eq:epsln}
\begin{array}{c}
\xymatrix@C=1.5em@R=1.5em{
\As
\ar[rr]|-@{|}^{\du{X}}
\ar[ddrr]|-@{|}_{\Hom_{\As}}^{\swnt\epsln_X}      &
&
\One \ar[dd]|-@{|}^X   \\
&
&
\\
&
&
\As
}
\end{array}
\end{align}
whose $(a, b)$-component
\[
X(a) \otimes [\As^\op, \V](X, \As(-, b)) \to \As(a, b)
\]
is defined in the case $\V = \Set$ by $(x, \xi) \mapsto \xi_b(x)$, and by
the obvious generalization for arbitrary $\V$. Equivalently, using the
second of the isomorphisms~\eqref{eq:etimes-adjn}, $\epsln_X$ is the map $X
\etimes \du{X} \to \Hom_{\As}$ corresponding to the identity on $\du{X}$.

The result is that $\epsln_X$ exhibits $\du{X}$ as the right Kan lift of
$\Hom_{\As}$ through $X$ in $\V\hyph\Prof$. (That is, the pair $(\du{X},
\epsln_X)$ is terminal of its type.) This follows from the second
adjointness relation in~\eqref{eq:prof-adj} on taking $P = X$ and $R =
\Hom_{\As}$.

Dually, for $Y \from \A \to \V$, a similarly defined transformation
$\epsln_Y \from \du{Y} \pof Y \to \Hom_\A$ exhibits $\du{Y}$ as the right
Kan extension of $\Hom_\A$ along $Y$ in $\V\hyph\Prof$.

\section{Examples of conjugacy}
\label{sec:eg-conj}

We list some examples of conjugacy, beginning with unenriched categories.

\begin{example}
\label{eg:conj-small-discrete}
Let $\As$ be a small discrete category, $Y \from \As \to \Set$, and $a \in
\As$. Then
\[
\du{Y}(a)
=
\begin{cases}
1               &\text{if } Y(b) = \emptyset \text{ for all } b \neq a  \\
\emptyset       &\text{otherwise.}
\end{cases}
\]
Thus, writing 
\begin{align}
\label{eq:defn-supp}
\supp Y = \{ b \in \As \such Y(b) \text{ is nonempty}\}, 
\end{align}
we have
\[
\du{Y} 
\iso
\begin{cases}
1               & \text{if } Y \iso 0 \\
\As(-, a)       & 
\text{if } 
\supp Y = \{a\} \\
0               & \text{otherwise.}
\end{cases}
\]
\end{example}

\begin{example}
\label{eg:conj-group}
Let $G$ be a group, seen as a one-object category. A functor $X \from G^\op \to
\Set$ is a right $G$-set, and the unique representable such functor is
$G_r$, the set $G$ acted on by $G$ by right multiplication. Thus, $\du{X}$
is the left $G$-set of $G$-equivariant maps $G \to G_r$. We now compute
$\du{X}$ explicitly.

First suppose that the $G$-set $X$ is nonempty, transitive and \demph{free}
(for $g \in G$, if $xg = x$ for some $x$ then $g = 1$). Then $X \iso G_r$,
so $\du{X}$ is isomorphic to $G_\ell$, the set $G$ acted on by the group
$G$ by left multiplication.

Next suppose that $X$ is nonempty and transitive but not
free. Choose $x \in X$ and $1 \neq g \in G$ such that $xg = x$. Any
equivariant $\alpha \from X \to G_r$ satisfies $\alpha(x) = \alpha(xg) =
\alpha(x)g$, a contradiction since $g \neq 1$. Hence $\du{X} = \emptyset$.

Finally, take an arbitrary $G$-set $X$. It is a coproduct $\sum_{i \in I}
X_i$ of nonempty transitive $G$-sets, so by adjointness, $\du{X} = \prod_{i
\in I} X_i^\vee$. By the previous paragraph, $\du{X}$ is empty unless every
orbit $X_i$ is free, or equivalently unless $X$ is free. If $X$ is free
then $X$ is the copower $I \times G_r$ and $\du{X}$ is the power
$G_\ell^I$. But $I \iso X/G$, so
\[
\du{X}
\iso
\begin{cases}
G_\ell^{X/G}    &\text{if } X \text{ is free}   \\
\emptyset       &\text{otherwise.}
\end{cases}
\]
\end{example}

\begin{example}
\label{eg:conj-poset}
Let $A$ be a partially ordered set regarded as a category, and let $X
\from A^\op \to \Set$. The set $\supp X \sub A$
(equation~\eqref{eq:defn-supp}) is downwards closed, and when $X = A(-,
a)$, it is $\dn a = \{ b \in A \such b \leq a\}$. Now
\[
\du{X}(a) \iso
\begin{cases}
1               &\text{if } \supp X \sub \dn a  \\
\emptyset       &\text{otherwise}.
\end{cases}
\]
Of course, the dual result also holds, involving $\up a = \{ b \in A \such b
\geq a \}$.

A $\Set$-valued functor on a category is \demph{subterminal} if it is a
subobject of the terminal functor, or equivalently if all of its values are
empty or singletons. Subterminal functors $A \to \Set$ correspond via
$\supp$ to upwards closed subsets of $A$. The conjugate of any functor $X
\from A^\op \to \Set$ is subterminal, corresponding to the upwards closed
set of upper bounds of $\supp X$ in $A$.
\end{example}

\begin{example}
\label{eg:conj-enr-poset}
Write $\Two = (0 \to 1)$ with $\min$ as monoidal structure. A
small $\Two$-category $A$ is a partially ordered set (up to equivalence),
and a $\Two$-functor $X \from A^\op \to \Two$ amounts to a downwards closed
subset of $A$, namely, $\{ a \in A \such X(a) = 1\}$.  Dually, a
$\Two$-functor $A \to \Two$ is an upwards closed subset of~$A$.

From this perspective, the conjugacy adjunction is as follows: for
a downwards closed set $X \subseteq A$, the upwards closed set $\du{X}$ is
the set of upper bounds of $X$, and dually.
\end{example}

\begin{example}
\label{eg:conj-Ab}
Write $\Ab$ for the category of abelian groups. A one-object $\Ab$-category
$R$ is a ring, and an $\Ab$-functor $R^\op \to \Ab$ is a right $R$-module.
The unique representable on $R$ is $R_r$, the abelian group $R$ regarded as
a right $R$-module.  Thus, the conjugate of a right module $M$ is
\[
\du{M} = \Mod_R(M, R_r)
\]
with the left module structure induced by the left action of $R$ on
itself. When $R$ is a field, $\du{M}$ is the dual of the vector space $M$.
\end{example}

\begin{example}
\label{eg:conj-metr}
Consider the ordered set $([0, \infty], \geq)$ with its additive monoidal
structure. This is a monoidal closed category, the internal hom $[x, y]$ being
the truncated difference
\[
y \truncsub x
=
\max\{y - x, 0\}.
\]
Lawvere~\cite{LawvMSG} famously observed that a $[0, \infty]$-category is a
generalized metric space, `generalized' in that distances need not be
symmetric or finite, and distinct points can be distance $0$ apart.

Let $A = (A, d)$ be a generalized metric space. A $[0, \infty]$-functor
$A^\op \to [0, \infty]$ is a function $f \from A \to [0, \infty]$ such that
\[
f(a) \truncsub f(b) \leq d(a, b)
\]
for all $a, b \in A$. Its conjugate $\du{f} \from A \to [0, \infty]$ is
defined by 
\[
\du{f}(a) = \sup_{b \in A} \bigl( d(b, a) \truncsub f(b) \bigr).
\]
\end{example}

\section{Conjugacy for general categories}
\label{sec:conj-gen}

To define conjugacy on a general category requires more delicacy than on a
small category. The reader who wants to get on to the reflexive completion
can ignore this section for now. However, because of the phenomenon noted
in Example~\ref{eg:compl-7}, the \emph{theory} of the reflexive completion
ultimately requires this more general definition of conjugacy: it is not
possible to confine oneself to small categories only.

The following example shows that the definition of conjugacy for small
categories cannot be extended verbatim to large categories.

\begin{example}
\label{eg:no-conj}
Let $\mathcal{C}$ be a proper class. Let $\cat{A}$ be the category obtained
by adjoining to the discrete category $\mathcal{C}$ a further object $z$
and maps $p^0_c, p^1_c \from z \to c$ for each $c \in \mathcal{C}$. Let $Y
\from \cat{A} \to \Set$ be the functor defined by
\[
Y(a) = 
\begin{cases}
1               &\text{if } a \in \mathcal{C}   \\
\emptyset       &\text{if } a = z.
\end{cases}
\]
A natural transformation $Y \to \cat{A}(z, -)$ is a choice of element of
$\{p^0_c, p^1_c\}$ for each $c \in \mathcal{C}$. There is a proper
class of such transformations, so there is no $\Set$-valued functor $\du{Y}
\from \cat{A}^\op \to \Set$ defined by $\du{Y}(a) = [\cat{A}, \Set](Y,
\cat{A}(a, -))$.
\end{example}

Since not every functor has a conjugate, we restrict ourselves to a class
of functors that do. These are the small functors introduced by Ulmer
(\cite{Ulme}, Remark~2.29). We briefly review them now, referring to Day
and Lack~\cite{DaLa} for details.

Again we work over a complete and cocomplete symmetric monoidal closed
category $\V$, understanding all categories, functors, etc., to be
$\V$-enriched. 

For a category $\A$, a functor $\A \to \V$ is \demph{small} if it can
expressed as a small colimit of representables, or equivalently if it is
the left Kan extension of its restriction to some small full subcategory of
$\A$, or equivalently if it is the left Kan extension of some $\V$-valued
functor on some small category $\Bs$ along some functor $\Bs \to \A$.

\begin{example}
When $\A$ is small, every functor $\A \to \V$ is small.
\end{example}

\begin{example}
Taking $\V = \Set$, the constant functor $1$ on a large discrete category
is not small; nor is the functor $Y$ of Example~\ref{eg:no-conj}.
\end{example}

\begin{example}
\label{eg:small-order}
For later purposes, let us consider an ordered class $A$ and a subterminal
functor $X \from A^\op \to \Set$ (as defined in
Example~\ref{eg:conj-poset}).  Then $X$ is small if and only if
there is some small $K \sub \supp X$ such that for all $a \in \supp
X$, the poset $K \cap \up a$ is connected (and in particular,
nonempty). This follows from the definition of a small functor
as one that is the left Kan extension of its restriction to some small full
subcategory. 
\end{example}

For arbitrary functors $X, X' \from \A^\op \to \V$, the $\V$-natural
transformations $X \to X'$ do not always define an object of $\V$, as
Example~\ref{eg:no-conj} shows in the case $\V = \Set$.  But when $X$ is
small, they do: it is the (possibly large) end
\begin{align}
\label{eq:VNat}
\VNat(X, X') = \int_a [X(a), X'(a)] \,\in \V.
\end{align}
To see that this end exists, first note that
by smallness of $X$, we can choose a small full subcategory $\C$ of
$\A$ such that $X$ is the left Kan extension of its restriction to
$\C$. Since $\C$ is small and $\V$ has small limits, the functor
$\V$-category $[\C^\op, \V]$ exists, and the universal property of Kan
extensions implies that $[\C^\op, \V]\bigl(X|_\C, X'|_\C\bigr)$ is the
end~\eqref{eq:VNat}. 

In particular, the small functors $\A^\op \to \V$ form a
$\V$-category $\cocmp{\A}$. We also write $\cmp{\A}$ for the \emph{opposite}
of the $\V$-category of small functors $\A \to \V$. When $\A$ is
small,
\[
\cocmp{\A} = [\A^\op, \V],
\qquad
\cmp{\A} = [\A, \V]^\op.
\]
When $\A$ is large, the right-hand sides are in general undefined as
$\V$-categories. In the case $\V = \Set$, the right-hand sides can be
interpreted as categories that are not locally small, but typically
\[
\cocmp{\A} \subsetneq [\A^\op, \Set],
\qquad
\cmp{\A} \subsetneq [\A, \Set]^\op.
\]

A small colimit of small $\V$-valued functors is small, so the
$\V$-category $\cocmp{\A}$ has small colimits, computed pointwise. Indeed,
it is the free cocompletion of $\A$: the Yoneda embedding $\A \incl
\cocmp{\A}$ is the initial functor (in a 2-categorical sense) from $\A$ to
a category with small colimits. Dually, $\cmp{\A}$ is the free
completion of~$\A$.

\begin{remark}
\label{rmk:proper}
Isbell used a different size condition, defining a $\Set$-valued functor to
be \demph{proper} if it admits an epimorphism from a small coproduct of
representables (\cite{IsbeAS}, Section~1). (Freyd later called such
functors `petty' \cite{FreySNC}.) Properness is a weaker condition than
smallness, but the universal properties of $\cocmp{\A}$ and $\cmp{\A}$ make
smallness a natural choice, and it generalizes smoothly to arbitrary $\V$.
\end{remark}

\begin{defn}
\label{defn:rep-small}
A functor $F \from \A \to \B$ is \demph{representably small} if for
each $b \in \B$, the functor 
\[
\nv{F}(b) = \B(F-, b) \from \A^\op \to \V
\]
is small, and \demph{corepresentably small} if for each $b \in \B$, 
\[
\conv{F}(b) = \B(b, F-) \from \A \to \V
\]
is small. (This is dual to the convention in Section~8 of Day and
Lack~\cite{DaLa}.) 
\end{defn}

Thus, a representably small functor $F \from \A \to \B$ induces a nerve
functor $\nv{F} \from \B \to \cocmp{\A}$, and dually.

\begin{lemma}
\label{lemma:small-comp}
Let $\A \toby{F} \B \toby{G} \C$ be functors. If $F$ and $G$ are 
representably small then so is $GF$, and dually for corepresentably small.
\end{lemma}

\begin{proof}
Suppose that $F$ and $G$ are representably small, and let $c \in \C$. We
must show that $\C(GF-, c)$ is small. By hypothesis, $\C(G-, c)$ is a small
colimit of representables, say $\C(G-, c) = W \mathbin{*} \B(-, D)$ where
$\scat{I}$ is a small category, $W \from \scat{I}^\op \to \V$ and $D \from
\scat{I} \to \B$. Then $\C(GF-, c) = W \mathbin{*} \B(F-, D)$, which by
hypothesis is a small colimit of small functors, hence small.
\end{proof}

This completes our review of smallness. Now let $\A$ be a category. The
\demph{conjugate} of a small functor $X \from \A^\op \to \V$ is the functor
$\du{X}\from \A \to \V$ defined by
\[
\du{X}(a) = \cocmp{\A}(X, \A(-, a)).
\]
Since $\cocmp{\A^\op} = \bigl(\cmp{\A}\,\bigr)^\op$, this implies that the
conjugate of a small functor $Y \from \A \to \V$ is the functor
$\du{Y} \from \A^\op \to \V$ defined by
\[
\du{Y}(a) = \cmp{\A}(Y, \A(a, -)).
\]

The conjugate of a small functor need not be small:

\begin{example}
\label{eg:conj-large-discrete}
Let $\A$ be a discrete category on a proper class of objects.  The small
functors $Y \from \A \to \Set$ are precisely those such that $\supp Y$ is
small. So the initial (empty) functor $0 \from \A \to \Set$ is small, but
its conjugate is the terminal functor $1$, which is not small.
\end{example}

When $\V = \Set$, conjugacy defines functors
\[
\cocmp{\A} \to [\A, \Set]^\op,
\qquad
\cmp{\A} \to [\A^\op, \Set],
\]
whose codomains are in general not locally small. For a general $\V$ and
$\A$, conjugacy is still contravariantly functorial in $X$ and $Y$, but
there are no $\V$-categories $[\A, \V]^\op$ and $[\A^\op, \V]$ to act as
the codomains. So it no longer makes sense to speak of a 
conjugacy \emph{adjunction}. However, we do have the following.

\begin{lemma}
\label{lemma:large-adj}
Let $\A$ be a category. Then
\[
\VNat(X, \du{Y}) \iso \VNat(Y, \du{X})
\]
naturally in $X \in \cocmp{\A}$ and $Y \in \cmp{\A}$.
\end{lemma}

Since $X$ and $Y$ are small, each side of the claimed isomorphism is a
well-defined object of $\V$ (equation~\eqref{eq:VNat}).

\begin{proof}
It is routine to verify that each side is naturally isomorphic to
$\VNat(X \etimes Y, \Hom_\A)$, where $\etimes$ was defined
in~\eqref{eq:external}.
\end{proof}

The isomorphism of Lemma~\ref{lemma:large-adj} gives rise in the usual way
to a canonical map $\eta_X \from X \to \ddu{X}$ whenever $X \from \A^\op
\to \V$ is a small functor such that $\du{X}$ is also small. Dually, for
any small functor $Y \from \A \to \V$ with small conjugate, there is a
canonical map $\eta_Y \from Y \to \ddu{Y}$.

\begin{remark}
\label{rmk:eta-unamb}
The reuse of the letter $\eta$ is not an abuse, in that $\eta_X$
is the same whether $X$ is regarded as a contravariant functor on $\A$ or a
covariant functor on $\A^\op$. (Compare Remark~\ref{rmk:unamb}.)
\end{remark}

In the case $\V = \Set$, the unit transformation $\eta$ can be described
explicitly as follows. Let $X \from \A^\op \to \Set$ be a small functor
with small conjugate. Let $a \in \A$ and $x \in X(a)$. Then $\eta_{X, a}(x)
\in \ddu{X}(a)$ is the natural transformation
\[
\eta_{X, a}(x) \from \du{X} \to \A(a, -)
\]
that evaluates at $x$: its component at $b \in \A$ is the function
\[
\begin{array}{ccc}
\cocmp{\A}(X, \A(-, b)) &\to            &\A(a, b)       \\
\xi                     &\mapsto        &\xi_a(x).
\end{array}
\]

\begin{remark}
\label{rmk:gentle}
Define a category $\A$ to be \demph{gentle} if $\cocmp{\A}$ is complete and
$\cmp{\A}$ is cocomplete. Small categories are certainly gentle. Day and
Lack proved that $\cocmp{\A}$ is complete if $\A$ is (Corollary~3.9
of~\cite{DaLa}), so by duality, any complete and cocomplete category is
also gentle. On the other hand, a large discrete category $\A$ is not
gentle, as $\cocmp{\A}$ has no terminal object.

For a gentle category $\A$, the conjugate of a small functor on $\A$ is
again small, so that conjugacy defines a genuine adjunction between
$\cocmp{\A}$ and $\cmp{\A}$. This was shown by Day and Lack in Section~9
of~\cite{DaLa}. 

In fact, when $\A$ is complete and cocomplete, the conjugate of a small
functor $\A \to \V$ or $\A^\op \to \V$ is not only small but
representable. Indeed, if $X$ is a small colimit of representables
$\cat{A}(-, a_i)$ then $\du{X}$ is represented by the corresponding colimit
of the $a_i$, and dually.  Specializing further, if $\A$ is totally
complete and totally cocomplete then the conjugate of \emph{any} functor on
$\A$ (not necessarily small) is representable, the representing object of
$\du{X}$ being the image of $X$ under the left adjoint to the Yoneda
embedding. (We thank the referee for these observations.)
\end{remark}

\section{The reflexive completion}
\label{sec:completion}

The reflexive completion of a category was first defined by Isbell
(Section~1 of~\cite{IsbeAS}), for unenriched categories. We consider it for
categories enriched in a complete and cocomplete symmetric monoidal closed
category $\V$, beginning with the case of small $\V$-categories and then
generalizing to arbitrary $\V$-categories. For small categories over $\V =
\Set$, our definition is precisely Isbell's. For general categories over
$\Set$, there is the set-theoretic difference that our definition uses small
functors where his used proper functors (Remark~\ref{rmk:proper}).

Recall that every adjunction
\[
\xymatrix{
\C \ar@/^/[r]^F & \cat{D}\ar@/^/[l]^G_{\perp}
}
\]
between $\V$-categories restricts canonically to an equivalence between
full subcategories of $\C$ and $\cat{D}$ (Lambek and
Rattray~\cite{LaRaLDA}, Theorem~1.1). The subcategory of $\C$ consists
of those objects $c$ for which the unit map $c \to GFc$ (in the underlying
category of $\C$) is an isomorphism, and dually for $\cat{D}$. We call
either of these equivalent subcategories the \demph{invariant part} of the
adjunction.

The \demph{reflexive completion} $\refl(\A)$ of a small $\V$-category $\A$
is the invariant part of the conjugacy adjunction
\[
\xymatrix{
\cocmp{\A} \ar@/^/[r]^\vee & \cmp{\A}.\ar@/^/[l]^{\vee^\op}_{\perp}
}
\]
When $\refl(\A)$ is seen as a full subcategory of $\cocmp{\A}$, it consists
of those functors $X \from \A^\op \to \V$ such that the unit map $\eta_X
\from X \to \ddu{X}$ is an isomorphism; such functors $X$ are called
\demph{reflexive}. Dually, $\refl(\A)$ can be seen as the full subcategory
of $\cmp{\A}$ consisting of the reflexive functors $\A \to \V$. 

Now let $\A$ be an any $\V$-category, not necessarily small. To define
reflexivity of a functor $X$ on $\A$, we need $\ddu{X}$ to be defined, so
we ask that $X$ and $\du{X}$ are small.

\begin{defn}
A functor $X \from \A^\op \to \V$ is \demph{reflexive} if $X \in
\cocmp{\A}$, $\du{X} \in \cmp{\A}$, and the canonical natural
transformation $\eta_X \from X \to \ddu{X}$ is an isomorphism. 
\end{defn}

This extends the earlier definition for small $\A$. Although conjugacy for
an arbitrary $\A$ does not define an adjunction between $\cocmp{\A}$ and
$\cmp{\A}$, it still induces an equivalence between the full subcategory of
$\cocmp{\A}$ consisting of the reflexive functors $\A^\op \to \V$ and the
full subcategory of $\cmp{\A}$ consisting of the reflexive functors $\A \to
\V$. The \demph{reflexive completion} $\refl(\A)$ of $\A$ is either of
these equivalent categories. 

In the case $\V = \Set$, we have the concrete description of $\eta_X$ given
after Lemma~\ref{lemma:large-adj}. It implies that $X \in \cocmp{A}$ is
reflexive if and only if for each $a \in \A$, every element of $\ddu{X}(a)$
is evaluation at a unique element of $X(a)$.

\begin{remark}
\label{rmk:refl-dual}
By Remark~\ref{rmk:eta-unamb}, whether a $\V$-valued functor is reflexive does
not depend on whether it is considered as a covariant functor on its domain
or a contravariant functor on the opposite of its domain. It follows that
$\refl(\A^\op) \eqv \refl(\A)^\op$ for all $\V$-categories $\A$.
\end{remark}

\begin{example}
\label{eg:rep-refl}
Let $\A$ be a $\V$-category. For each $a \in \A$,
\[
\du{\A(-, a)} \iso \A(a, -),
\qquad
\du{\A(a, -)} \iso \A(-, a),
\]
and the unit map $\A(-, a) \to \ddu{\A(-, a)}$ is an isomorphism. Hence
representables are reflexive.
\end{example}

The image of the Yoneda embedding $\A \incl \cocmp{\A}$ therefore lies in
$\refl(\A)$, when the latter is seen as a subcategory of $\cocmp{\A}$. A
dual statement holds for $\cmp{\A}$. There is, then, an unambiguous Yoneda
embeddding
\[
J_{\A} \from \A \to \refl(\A)
\]
such that the diagram of full and faithful functors
\begin{align}
\label{eq:inclusions}
\begin{array}{c}
\xymatrix@C+3em{
        &       &\cocmp{\cat{A}}        \\
\cat{A} 
\ar@<1ex>@{^{(}->}[rru] 
\ar@{^{(}->}[r]|{J_{\A}} 
\ar@<-1ex>@{^{(}->}[rrd]    &
\refl(\cat{A}) 
\ar@{^{(}->}[ru] 
\ar@{^{(}->}[rd]
        &       \\
        &       &\cmp{\cat{A}}
}
\end{array}
\end{align}
commutes.

\section{Examples of reflexive completion}
\label{sec:eg-compl}

We begin with unenriched examples.

\begin{example}
\label{eg:compl-empty}
Let $\Zero$ denote the empty category.  Then $[\Zero^\op, \Set]$ and
$[\Zero, \Set]^\op$ are both the terminal category $\One$, so $\refl(\Zero)
\eqv \One$.  In particular, the reflexive completion of a category need
not be equivalent to its Cauchy completion.
\end{example}

\begin{example}
\label{eg:compl-one}
The conjugacy adjunction for the terminal category $\One$ consists of the
functors $\Set \oppairu \Set^\op$ with constant value $1$, giving
$\refl(\One) \eqv \One$.
\end{example}

\begin{example}
\label{eg:compl-small-discrete}
Let $\As$ be a small discrete category with at least two objects.  By
Example~\ref{eg:conj-small-discrete}, for $Y \from \As \to \Set$,
\[
\ddu{Y} 
\iso
\begin{cases}
0               & \text{if } Y \iso 0 \\
\As(a, -)       & 
\text{if } 
\supp Y = \{a\} \\
1               & \text{otherwise.}
\end{cases}
\]
Hence the reflexive functors $\As \to \Set$ are those that are initial,
terminal or representable. (Contrast this with Example~\ref{eg:compl-one},
in which the initial functor $\One \to \Set$ is not reflexive.) It follows
that $\refl(\As)$ is $\As$ with initial and terminal objects adjoined.
\end{example}

\begin{example}
\label{eg:compl-large-discrete}
Now let $\A$ be a large discrete category. As observed in
Example~\ref{eg:conj-large-discrete}, the conjugate of the small functor $0
\from \A \to \Set$ is the non-small functor $1\from \A^\op \to \Set$. Hence
neither $0$ nor $1$ is reflexive. The same argument as in
Example~\ref{eg:compl-small-discrete} then shows that the only reflexive
functors on $\A$ are the representables.  Thus, unlike in the small case,
the reflexive completion of a large discrete category is itself.
\end{example}

\begin{example}
\label{eg:compl-groups}
Let $G$ be a group, regarded as a one-object category. If $G$ is trivial
then $\refl(G) \eqv \One$ by Example~\ref{eg:compl-one}. Suppose not. 

Let $X$ be a right $G$-set.  In Example~\ref{eg:conj-group}, we showed that
\[
\du{X}
\iso
\begin{cases}
G_\ell^{X/G}    &\text{if } X \text{ is free}   \\
\emptyset       &\text{otherwise,}
\end{cases}
\]
and of course a similar result holds for left $G$-sets.  If $X$ is not
free then $\ddu{X} = 1$, so the only non-free reflexive $G$-set is the
terminal $G$-set $1$.

Now suppose that $X$ is free. If $X$ is empty then $\ddu{X} = \du{1} =
\emptyset$ (using the nontriviality of $G$ in the second equality), so $X$
is reflexive. Assume now that $X$ is nonempty, write $S = X/G$, and choose
$s_0 \in S$.  The left $G$-action on $\du{X} \iso G_\ell^S$ is free, and
each orbit contains exactly one element whose $s_0$-component is the
identity element of $G$, so $\du{X}$ has $|G|^{|S \setminus \{s_0\}|}$
orbits. Writing $|S| - 1$ for $|S \setminus \{s_0\}|$, we conclude that
$\du{X}$ is a free $G$-set with
\[
|\du{X}/G| = |G|^{|S|-1}.
\]
Repeating the argument in the dual situation then gives
\[
|\ddu{X}/G| = |G|^{|G|^{|S|-1}-1}.
\]
Hence if $X$ is reflexive,
\[
|S| =  |G|^{|G|^{|S|-1}-1}.
\]
By elementary arguments, this implies that $|S| = 1$ (in which case $X$ is
representable) or $|G| = |S| = 2$.  Hence when $|G| > 2$, the only
reflexive right $G$-sets are $\emptyset$, $1$ and $G_r$.

The remaining case is where $G$ is the two-element group and $X$ is the
free $G$-set on two generators. A direct calculation shows that 
\[
G_r + G_r \iso G_r \times G_r
\]
in $[G^\op, \Set]$. Since $G$ is abelian, the same is true for
$G_\ell$. Now using the adjoint property of conjugates,
\begin{align*}
\du{X}  &
\iso   
\du{(G_r + G_r)} 
\iso
G_\ell \times G_\ell    
\iso
G_\ell + G_\ell,        \\
\ddu{X} &
\iso    
\du{(G_\ell + G_\ell)}  
\iso
G_r \times G_r  
\iso
G_r + G_r,
\end{align*}
giving $\ddu{X} \iso X$. We claim that $X$ is reflexive, that is, the unit
map $\eta_X \from X \to \ddu{X}$ is an isomorphism.  One of the triangle
identities for the conjugacy adjunction implies that $\eta_{\du{W}}$ is
split monic for any $W \from G \to \Set$.  But $X \iso \du{(\du{X})}$, so $\eta_X$ is
an injection between finite sets of the same cardinality, hence bijective,
hence an isomorphism.

In summary, the reflexive completion of a group $G$ is as follows:
\begin{itemize}
\item 
if $|G| = 1$ then $\refl(G) \iso G$;

\item
if $|G| = 2$ then $\refl(G)$ is the full subcategory of the category
of $G$-sets consisting of the initial $G$-set, the terminal
$G$-set, the representable $G$-set $G_r$, and the four-element $G$-set
$G_r + G_r \iso G_r \times G_r$.  

\item 
if $|G| > 2$ then $\mathcal{R}(G)$ is the full subcategory of the category
of right $G$-sets consisting of the initial $G$-set, the terminal $G$-set,
and the representable $G$-set. It is equivalent to $G$ with initial and
terminal objects adjoined.
\end{itemize}
\end{example}

\begin{remark}
For a $\Set$-valued functor $X$, the sequence $X, \du{X}, \ddu{X},
\dddu{X}, \ldots$ need not ever repeat itself.  For consider functors on
the three-element group $C_3$. By Example~\ref{eg:compl-groups}, the
conjugate of the free $C_3$-set $n \times C_3$ on $n$ elements is $3^{n -
1} \times C_3$.  Since $3^{n - 1} > n$ for all $n \geq 2$, no two of the
iterated conjugates of $2 \times C_3$ are isomorphic.
\end{remark}

\begin{example}
Let $M$ be the two-element commutative monoid whose non-identity element
$e$ is idempotent.  A covariant or contravariant functor from $M$ to $\Set$
amounts to a set $X$ together with an idempotent endomorphism $f$.  The
representable such functor corresponds to the set $M = \{\id,e\}$ together
with the endomorphism with constant value $e$.

Given a pair $(X, f)$, write $X_0 = \im f$ and $X_1 = X \without X_0$. A
routine calculation shows that
\[
\bigl|\bigl(\ddu{X}\bigr)_0\bigr| 
= 
1 
\qquad 
\text{and} 
\qquad 
\bigl| \bigl(\ddu{X}\bigr)_1 \bigr| 
=
2^{2^{|X_1|} - 1} - 1.
\]
So if $X$ is reflexive then $|X_0| = 1$ and $|X_1| = 2^{2^{|X_1|}-1}-1$, and
the latter equation implies that $|X_1|$ is $0$ or $1$. Thus, $X$ is either
the representable functor or the terminal functor on $M$. Both are indeed
reflexive.

The reflexive completion of $M$ is, therefore, the full subcategory of the
category of $M$-sets consisting of $M$ itself and $1$.  This is the free
category on a split epimorphism, and is the same as the Cauchy
completion of $M$.
\end{example}

\begin{example}
\label{eg:compl-7}
There is a $7$-element monoid whose reflexive completion is large. In
particular, the reflexive completion of a finite category need not even be
small.  This is an example of Isbell to which we return in
Example~\ref{eg:char-7}.
\end{example}

\begin{remark}
The Cauchy completion of a category $\A$ has a well-known concrete
description: an object is an object $a \in \A$ together with an idempotent
$e$ on $a$, a map $(a, e) \to (a', e')$ is a map $f \from a \to a'$ in $\A$
such that $e'fe = f$, composition is as in $\A$, and the identity on $(a,
e)$ is $e$. It follows that the Cauchy completion of a finite or small
category is finite or small, respectively. Example~\ref{eg:compl-7} implies
that the reflexive completion can have no very similar description.
\end{remark}

\begin{example}
\label{eg:compl-posets}
Let $A$ be a poset, regarded as a category.  The conjugacy adjunction of
$A$, when restricted to subterminal functors (Example~\ref{eg:conj-poset}),
is an adjunction
\[
\xymatrix{
\{\text{downwards closed subsets of } A\}
\ar@<.5ex>[r]^-\uparrow       &
\{\text{upwards closed subsets of } A\}^\op.
\ar@<.5ex>[l]^-\downarrow
}
\]
Here both sets are ordered by inclusion, and $\up S$ and $\dn S$ are the
sets of upper and lower bounds of a subset $S \subseteq A$,
respectively. The adjointness states that $X \sub \dn Y \iff Y
\sub \up X$.

A reflexive functor $A^\op \to \Set$ amounts to a downwards closed set $X
\sub A$ such that $X = \dn \up X$. The reflexive completion $\refl(A)$ is
the set of such subsets $X$, ordered by inclusion. This is the
\demph{Dedekind--MacNeille completion} of the poset $A$ (Definition~7.38 of
Davey and Priestley~\cite{DaPr}).  For example, the reflexive completion of
$(\Q, \leq)$ is the extended real line $[-\infty, \infty]$.

The Dedekind--MacNeille completion of a poset is always a complete
lattice. (And conversely, any complete lattice is the Dedekind--MacNeille
completion of itself.) However, the reflexive completion of an arbitrary
small category is far from complete, as Theorem~\ref{thm:rcn-lim} shows.
\end{example}

\begin{example}
\label{eg:compl-posets-2}
Now let $A$ be a poset, but regarded as a category enriched in $\Two$. By
Example~\ref{eg:conj-enr-poset} and the same argument as in
Example~\ref{eg:compl-posets}, the reflexive completion of $A$ as a
$\Two$-category is again its Dedekind--MacNeille completion.
\end{example}

\begin{example}
\label{eg:compl-ring}
As in Example~\ref{eg:conj-Ab}, let $\cat{V} = \Ab$, let $R$ be a ring, and
let $M$ be a right $R$-module.  The double conjugate of $M$ is its double dual
\[
\ddu{M} = {}_R\Mod(\Mod_R(M, R_r), R_\ell),
\]
so the notion of reflexive functor on $R$ coincides with the algebraists'
notion of reflexive module (as in Section~5.1.7 of McConnell and
Robson~\cite{McRo}).  Hence the reflexive completion of a ring, viewed as a
one-object $\Ab$-category, is its category of reflexive modules.  In
particular, the reflexive completion of a field $k$ is the category of
finite-dimensional $k$-vector spaces.
\end{example}

The rest of this section concerns the case $\V = [0, \infty]$
(Example~\ref{eg:conj-metr}), summarizing results from Willerton's
analysis~\cite{WillTSI} of the reflexive completion of a generalized metric
space $A$.  It follows from Example~\ref{eg:conj-metr} that $\refl(A)$ is
the set of distance-decreasing functions $f \from A^\op \to [0, \infty]$
such that
\[
f(c) = 
\sup_{b \in A} \inf_{a \in A}
\Bigl( d(c, b) \truncsub \bigl( d(a, b) \truncsub f(a) \bigr) \Bigr)
\]
for all $c \in A$, with metric 
\begin{align}
\label{eq:refl-met-dist}
d(f, g) = \sup_a \bigl(g(a) \truncsub f(a)\bigr).
\end{align}

The reflexive completion of $A$ consists of the $[0, \infty]$-valued
functors equal to their \emph{double} conjugate. But when $A$ is symmetric,
covariant and contravariant functors on $A$ can be identified, so we can
form the set
\[
T(A) = \{ \text{distance-decreasing functions } f \from A \to [0, \infty]
\text{ such that } \du{f} = f \}
\]
of $[0, \infty]$-valued functors equal to their \emph{single}
conjugate. Then $T(A)$, metrized as in equation~\eqref{eq:refl-met-dist},
is called the \demph{tight span} of $A$. Evidently $A \sub T(A) \sub
\refl(A)$. Both inclusions can be strict, as the following example shows.

\begin{example}
\label{eg:two-point}
Take the symmetric metric space $\{0, D\}$ consisting of two points
distance $D$ apart (Figure~\ref{fig:two-point}). It follows from the
description above that $\refl(\{0, D\})$ is the set $[0, D]^2$ with metric
\[
d\bigl((t_1, t_2), (u_1, u_2)\bigr) =
\max \{ u_1 \truncsub t_1, u_2 \truncsub t_2 \}
\]
(Example~3.2.1 of~\cite{WillTSI}).  The Yoneda embedding $\{0, D\} \to [0,
D]^2$ is given by $0 \mapsto (0, D)$ and $D \mapsto (D, 0)$.
\begin{figure}
\setlength{\unitlength}{1mm}
\begin{picture}(120,35)
\cell{60}{5}{b}{\includegraphics[height=30\unitlength]{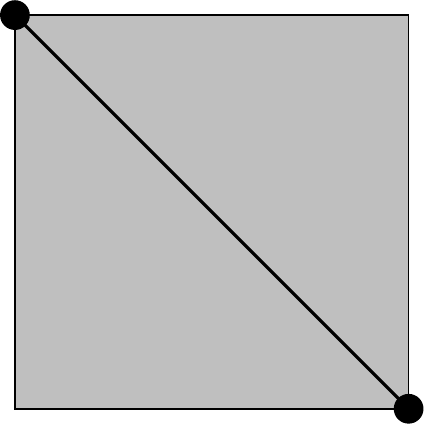}}
\cell{40}{35}{t}{$(0, D)$}
\cell{80}{5}{b}{$(D, 0)$}
\cell{62}{22}{c}{\rotatebox{-45}{$T(\{0, D\})$}}
\cell{60}{0}{b}{$\refl(\{0, D\})$}
\end{picture}%
\caption{The symmetric metric space $\{0, D\}$ embedded into its tight span
$T(\{0, D\})$ and reflexive completion $\refl(\{0, D\})$
(Example~\ref{eg:two-point}).} 
\label{fig:two-point}
\end{figure}
The tight span $T(\{0, D\})$ is the interval $[0, D]$ with its usual
metric, and embeds in $\refl(\{0, D\})$ as shown.
\end{example}

The tight span construction has been discovered independently several
times, as recounted in the introduction of~\cite{WillTSI}. That the form
given here is equivalent to other forms of the definition was established
by Dress (Section~1 of~\cite{DresTTE}).  The first to discover it was
Isbell~\cite{IsbeSTI}, who called it the `injective envelope', $T(A)$ being
the smallest injective metric space containing $A$.  But Isbell does not
seem to have noticed the connection with Isbell conjugacy.

The tight span is only defined for symmetric spaces, and is itself
symmetric (not quite trivially).  On the other hand, the reflexive
completion of a symmetric space need not be symmetric, as the
two-point example shows.  Theorem~4.1.1 of~\cite{WillTSI} states that the
tight span is the symmetric part of the reflexive completion:

\begin{thm}[Willerton]
Let $A$ be a symmetric metric space.  Then the tight span $T(A)$ is
the largest symmetric subspace of $ \refl(A)$ containing $A$.
\end{thm}

Here `largest' is with respect to inclusion. A nontrivial corollary is that
$\refl(A)$ \emph{has} a largest symmetric subspace containing $A$.

Finally, reflexive completion of metric spaces has arisen in fields far from
category theory. Pursuing a project in combinatorial optimization, Hirai
and Koichi~\cite{HiKo} defined the `directed tight span' of a generalized
metric space. As Willerton showed (Theorem~4.2.1 of~\cite{WillTSI}), it is
exactly the reflexive completion.

\section{Dense and adequate functors}
\label{sec:density}

Here we gather results on dense and adequate functors that will be used
later to characterize the reflexive completion. Some can be found in
Isbell's or Ulmer's foundational papers~\cite{IsbeAS,Ulme} or in Chapter~5
of Kelly~\cite{KellBCE}, while some appear to be new. For the rest of this
work, we restrict to unenriched categories, although much of what we do can
be extended to the enriched setting.

\begin{defn}
A functor $F \from \A \to \B$ is \demph{dense} if its nerve functor $\nv{F}
\from \B \to [\A^\op, \Set]$ (Definition~\ref{defn:rep-small}) is full and
faithful, and \demph{codense} if $\conv{F} \from \B \to [\A, \Set]^\op$ is
full and faithful. 

A functor is \demph{small-dense} if dense and representably small, and
\demph{small-codense} if codense and corepresentably small.
\end{defn}

\begin{remark}
It is a curious fact (not needed here) that while $F$ is dense if and only
if $\nv{F}$ is full and faithful, $F$ is full and faithful if and only if
$\nv{F}$ is dense.
\end{remark}

\begin{example}
\label{eg:adjt-dense}
Let $F \from \A \to \B$ be a functor with a right adjoint $G$. It is very
well known that $G$ is full and faithful if and only if the counit of the
adjunction is an isomorphism. Less well known, but already pointed out by
Ulmer in~1968 (Theorem~1.13 of~\cite{Ulme}), is that these conditions are
also equivalent to $F$ being dense. Thus, any functor $F$ with a full and
faithful right adjoint is dense. Indeed, $F$ is small-dense, since $\B(F-,
b)$ is representable for each $b \in \B$.
\end{example}

A functor $F \from \A \to \B$ is small-dense when $\nv{F}$
is full, faithful, and takes values in the category $\cocmp{\A}$ of small
functors $\A^\op \to \Set$. Then $\B$ embeds fully into $\cocmp{\A}$. When
$\A$ is small, every dense functor $\A \to \B$ is small-dense.

\begin{example}
The archetypal dense functor is the Yoneda embedding $\A \incl [\A^\op,
\Set]$, and the archetypal small-dense functor is the Yoneda embedding $\A
\incl \cocmp{\A}$.
\end{example}

A standard result is that $F \from \A \to \B$ is dense if and only if
every object of $\B$ is canonically a colimit of objects of the form
$Fa$; that is, for each $b \in \B$, the canonical cocone on the diagram
\[
(F \mathbin{\downarrow} b) \toby{\text{pr}} \A \toby{F} \B
\]
with vertex $b$ is a colimit cocone (Section~X.6 of
Mac~Lane~\cite{MacLCWM}). In terms of coends, this means that
\[
b \iso \int^a \B(Fa, b) \times Fa.
\]
That the Yoneda embedding is dense gives the \demph{density formula}
\[
X \iso \int^a X(a) \times \A(-, a)
\]
for functors $X \from \A^\op \to \Set$.

\begin{example}
\label{eg:dense-poset}
When $B$ is an ordered set (or class) and $A \sub B$ with the induced
order, the inclusion $A \incl B$ is dense if and only if it is
\demph{join-dense}: every element of $B$ is a join of elements of $A$.
\end{example}

\begin{lemma}
\label{lemma:sd-colim}
Let $F \from \A \to \B$ be a small-dense functor. Then for each $b \in \B$,
there is a small diagram $(a_i)_{i \in \scat{I}}$ in $\A$ such that $b \iso
\colim_i Fa_i$.
\end{lemma}

\begin{proof}
Let $b \in B$.  Since $F$ is representably small, we can choose a small
diagram $(a_i)$ in $\A$ such that $B(F-, b) \iso \colim_i \A(-, a_i)$. Then
by density of $F$ and the density formula,
\[
b 
\iso 
\int^a B(Fa, b) \times Fa
\iso
\int^{a, i} \A(a, a_i) \times Fa
\iso
\int^i Fa_i.
\]
\end{proof}

We will be especially interested in the (co)density of functors that are
full and faithful. Up to equivalence, such functors are inclusions of full
subcategories, which are called \demph{(co)dense} or \demph{small-(co)dense
subcategories} if the inclusion functor has the corresponding
property.

We now state some basic lemmas on full, faithful and dense functors,
beginning with one whose proof is immediate from the definitions.

\begin{lemma}
\label{lemma:ff}
Let $F \from \A \to \B$ be a full and faithful functor. Then the composite
\[
\A \toby{F} \B \toby{\nv{F}} [\A^\op, \Set]
\]
is canonically isomorphic to the Yoneda embedding.
\qed
\end{lemma}

The next lemma follows from the corollary to Proposition~3.2 in
Lambek~\cite{LambCC}, but we include the short proof for completeness.

\begin{lemma}
\label{lemma:dense-cts}
Every full and faithful dense functor preserves all (not just small) limits.
\end{lemma}

\begin{proof}
We use Lemma~\ref{lemma:ff}. Since $\nv{F}$ is full and faithful, it
reflects arbitrary limits; but the Yoneda embedding preserves them, so $F$
does too.
\end{proof}

The composite of two full and faithful dense functors need not be
dense. Isbell gave one counterexample (paragraph~1.2 of~\cite{IsbeAS}) and
Kelly gave another (Section~5.2 of~\cite{KellBCE}). Nevertheless:

\begin{lemma}
\label{lemma:dense-cocts-comp}
Let $\A \toby{F} \B \toby{G} \C$ be dense functors, and suppose 
that $G$ preserves arbitrary colimits. Then $GF$ is dense.
\end{lemma}

\begin{proof}
For $c \in \C$, we have canonical isomorphisms
\begin{align*}
c       &
\iso 
\int^b \C(Gb, c) \times Gb      & 
\text{($G$ is dense)}   \\
& 
\iso 
\int^b \C(Gb, c) \times G \biggl( \int^a \B(Fa, b) \times Fa\biggr)      & 
\text{($F$ is dense)} \\
& 
\iso 
\int^{a, b} \C(Gb, c) \times \B(Fa, b) \times GFa       & 
\text{($G$ preserves colimits)} \\
& 
\iso 
\int^a \C(GFa, c) \times GFa   & 
\text{(density formula),}
\end{align*}
so $GF$ is dense.
\end{proof}

\begin{defn}
\label{defn:adeq}
A functor is \demph{adequate} if it is full, faithful, dense and codense,
and \demph{small-adequate} if it is full, faithful, small-dense and
small-codense. 
\end{defn}

\begin{remark}
\label{rmk:dense-adeq}
Isbell's foundational paper~\cite{IsbeAS} considered adequacy only for full
subcategories. Up to equivalence, this amounts to working only with
functors that are full and faithful. For him, fullness and faithfulness
were implicit assumptions rather than explicit hypotheses.  He used
`left/right adequate' for what is now called dense/codense, and `adequate'
for dense and codense. The word `dense' was introduced later by
Ulmer~\cite{Ulme}, who extended the theory to arbitrary functors.
\end{remark}

\begin{example}
\label{eg:poset-adeq}
Let $f \from A \to B$ be an order-preserving map between partially ordered
classes. It is full and faithful if and only if it is an order-embedding
(that is, reflects the order relation), and by
Example~\ref{eg:dense-poset}, it is adequate if and only if it is also both
join-dense and meet-dense. Small-adequacy means that each element of $B$ is
a join of some \emph{small} family of elements of $\im f$, and similarly for
meets. 
\end{example}

\begin{lemma}
\label{lemma:adeq-comp}
The classes of adequate and small-adequate functors are each closed under
composition.
\end{lemma}

Isbell proved an analogue of this result for properly adequate functors
(statement~1.6 of~\cite{IsbeAS}), using a different argument.

\begin{proof}
Let $\A \toby{F} \B \toby{G} \C$ be adequate functors. Then $G$ preserves
arbitrary colimits by the dual of Lemma~\ref{lemma:dense-cts}, so $GF$ is
dense by Lemma~\ref{lemma:dense-cocts-comp}.  Dually, $GF$ is codense. So
$GF$ is adequate. If $F$ and $G$ are small-adequate then so is $GF$, by
Lemma~\ref{lemma:small-comp}.
\end{proof}

\begin{lemma}
\label{lemma:adeq-epi}
For adequate functors
$\xymatrix@1@C-2.5mm{ \A \ar[r]^F &\B \ar@<.5ex>[r]^G \ar@<-.5ex>[r]_{G'} &\C}$,
if $GF \iso G'F$ then $G \iso G'$.
\end{lemma}

\begin{proof}
We prove the stronger result that if $F$ is codense and $G$ and $G'$ are
full, faithful and dense then $GF \iso G'F \implies G \iso G'$. Indeed,
under these assumptions, Lemma~\ref{lemma:dense-cts} implies that $G$
preserves all limits, so by codensity of $F$,
\[
Gb      
\iso 
G\int_a [B(b, Fa), Fa] 
\iso
\int_a [B(b, Fa), GFa]
\]
naturally in $b \in B$ (where $[-, -]$ denotes a power). The same holds for
$G'$, so if $GF \iso G'F$ then $G \iso G'$.
\end{proof}

For full subcategories $\A \sub \B \sub \C$, if $\A$ is
dense in $\C$ then both intermediate inclusions are dense. More generally: 

\begin{lemma}
\label{lemma:dense-factors}
Let $\A \toby{F} \B \toby{G} \C$ be functors, with $G$ full and faithful.
\begin{enumerate}
\item
\label{part:df-d}
If $GF$ is dense then so are $F$ and $G$.

\item
\label{part:df-sd}
If $GF$ is small-dense then so is $F$.
\end{enumerate}
\end{lemma}

That $G$ is dense was asserted without proof in statement~1.1 of
Isbell~\cite{IsbeAS}. 

\begin{proof}
For~\bref{part:df-d}, to prove that $F$ is dense, note that the composite
functor 
\[
\begin{array}{ccccc}
\B      &\toby{G}       &\C     &\toby{\nv{GF}} &
[\A^\op, \Set]  \\
b       &\longmapsto    &Gb     &\longmapsto    &
\C(GF-, Gb) \iso \B(F-, b)
\end{array}
\]
is isomorphic to $\nv{F}$.  But both $G$ and $\nv{GF}$ are full and faithful, so
$\nv{F}$ is too.

To prove that $G$ is dense, let $c \in \C$. Then
\begin{align*}
c       &
\iso
\int^a \C(GFa, c) \times GFa    &
\text{($GF$ is dense)}  \\
&
\iso
\int^{a, b} \C(Gb, c) \times \B(Fa, b) \times GFa       &
\text{(density formula)}        \\
&
\iso
\int^b \C(Gb, c) \times \int^a \C(GFa, Gb) \times GFa   &
\text{($G$ is full and faithful)}       \\
&
\iso
\int^b \C(Gb, c) \times Gb      &
\text{($GF$ is dense)}
\end{align*}
naturally in $c$, as required.

For~\bref{part:df-sd}, suppose that $GF$ is small-dense. We must prove that
$F$ is representably small. Let $b \in \B$. Since $G$ is full and faithful,
$\B(F-, b) \iso \C(GF-, Gb)$, which is small since $GF$ is small-dense.
\end{proof}

\begin{propn}
\label{propn:incl-adeq}
For every category $\A$, the Yoneda embedding $J_{\A} \from \A \to
\refl(\A)$ is small-adequate. 
\end{propn}

\begin{proof}
We refer to diagram~\eqref{eq:inclusions}
(p.~\pageref{eq:inclusions}). Certainly $J_\A$ is full and
faithful. Lemma~\ref{lemma:dense-factors}\bref{part:df-sd} applied to $\A
\toby{J_\A} \refl(\A) \incl \cocmp{\A}$ implies that $J_\A$ is
small-dense. By duality, it is also small-codense.
\end{proof}

Lemma~\ref{lemma:dense-factors} and its dual immediately imply:

\begin{propn}
\label{propn:adeq-factors}
Let $\A \toby{F} \B \toby{G} \C$ be functors, with $G$ full and faithful.
If $GF$ is small-adequate then $G$ is adequate and $F$ is small-adequate. 
\qed
\end{propn}

Propositions~\ref{propn:incl-adeq} and~\ref{propn:adeq-factors} have the
following corollary. It is implicit in Section~1 of Isbell~\cite{IsbeAS},
modulo the difference in size conditions (Remark~\ref{rmk:proper}).

\begin{cor}[Isbell]
\label{cor:intermediate-adeq}
Let $\A$ be a category. For any full subcategory $\B$ of $\refl(\A)$
containing the representables, the inclusion $\A \incl \B$ is
small-adequate. 
\qed
\end{cor}

Proposition~\ref{propn:adeq-factors} does not conclude that $G$ must be
\emph{small}-adequate. Indeed, it need not be. This apparently
technical point becomes important later, so we give both a counterexample
and a sufficient condition for $G$ to be small-adequate.

\begin{example}
\label{eg:ordinals}
We exhibit ordered classes $A \sub B \sub C$ such that $A \incl
C$ is small-adequate (hence $A \incl B$ is too, by
Proposition~\ref{propn:adeq-factors}) but $B \incl C$ is not.

Let $P$ be the $5$-element poset shown in Figure~\ref{fig:ordinals}, and
let $Q$ be the ordered class of ordinals with a greatest element $\infty$
adjoined. Put
\[
C = P \times Q,
\quad
B = C \without \{(3, \infty)\},
\quad
A = C \without (\{3\} \times Q),
\]
with the product order on $C$ and the induced orders on $A, B \sub C$.  We
will prove that $A \incl C$ is small-adequate but $B \incl C$ is not
representably small, and, therefore, not small-adequate.

\begin{figure}
\setlength{\unitlength}{1mm}
\setlength{\fboxsep}{0mm}
\begin{picture}(120,50)
\cell{5}{25}{l}{\includegraphics[height=20\unitlength]{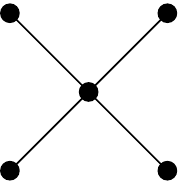}}
\cell{52}{25}{l}{\includegraphics[height=25\unitlength]{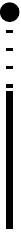}}
\cell{115}{25}{r}{\includegraphics[height=50\unitlength]{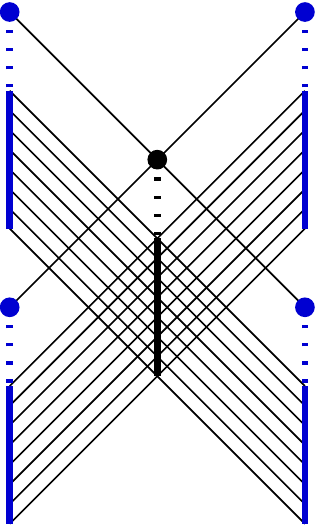}}
\cell{4}{16}{r}{$1$}
\cell{4}{34}{r}{$4$}
\cell{15}{28}{c}{$3$}
\cell{26}{16}{l}{$2$}
\cell{26}{34}{l}{$5$}
\cell{55}{36}{l}{$\infty$}
\cell{100}{38}{c}{\colorbox{white}{$(3, \infty)$}}
\cell{15}{8}{b}{$P$}
\cell{53}{8}{b}{$Q$}
\cell{100}{-1}{b}{$C = P \times Q$}
\end{picture}
\caption{The ordered classes of Example~\ref{eg:ordinals}, with $A \sub C$
shown in blue.}
\label{fig:ordinals}
\end{figure}

First we show that $A \incl C$ is dense, that is, $A$ is join-dense in $C$
(Example~\ref{eg:dense-poset}). Let $c \in C$. If $c \in A$ then $c$ is
trivially a join of elements of $A$. Otherwise, $c = (3, q)$ for some $q
\in Q$, and then $c = (1, q) \jn (2, q)$ with $(1, q), (2, q) \in A$.

Second, $A \incl C$ is codense by the same argument, because it
used nothing about the ordered class $Q$ and because $P \iso P^\op$.

Third, we show that $A \incl C$ is representably small; that is, for each
$c \in C$, the subterminal functor $A^\op \to \Set$ with support $A \cap
\dn c$ is small. Let $c \in C$. By Example~\ref{eg:small-order}, we must
find a small $K \sub A \cap \dn c$ such that for all $a \in A$, the poset
$K \cap \up a$ is connected.

If $c \in A$, we can take $K = \{a\}$. Otherwise, $c = (3, q)$ for some $q
\in Q$. Put $K = \{(1, q), (2, q)\}$. Given $a \in A \cap \dn c$, we may
suppose without loss of generality that $a = (1, q')$ for some $q' \leq q$,
and then $K \cap \up a$ is the connected poset $\{(1, q)\}$.

Fourth, $A \incl C$ is corepresentably small by the same argument, again
because it used nothing about $Q$ and because $P \iso P^\op$.

Finally, we show that $B \incl C$ is not representably small. In fact, we
show that the subterminal functor $B^\op \to \Set$ with support $B \cap \dn
(3, \infty)$ is not small. Suppose for a contradiction that it is
small. Then we can choose a small class
$
K 
\sub 
B \cap \dn (3, \infty) 
$
such that for all $b \in B$, the poset $K \cap \up b$ is connected. In
particular, every element of $B$ is less than or equal to some element of
$K$. Now for each ordinal $q$ we have $(3, q) \in B$, so $(3, q) \leq k$
for some $k \in K$, and then $k = (3, q')$ for some ordinal $q'$ with $q
\leq q'$. So if we put $Q' = \{ \text{ordinals } q' \such (3, q') \in K\}$
then $Q'$ is small (since $K$ is) and every ordinal is less than or equal to
some element of $Q'$. But there is no set of ordinals with this
property, a contradiction.
\end{example}

The following companion to Proposition~\ref{propn:adeq-factors} uses the
notion of gentle category from Remark~\ref{rmk:gentle}.

\begin{lemma}
\label{lemma:gentle-comp}
Let $\A \toby{F} \B \toby{G} \C$ be functors, with $G$ full and
faithful. Suppose that $\B$ is gentle. If $GF$ is small-adequate then so is
$G$. 
\end{lemma}

\begin{proof}
If $GF$ is adequate then $G$ is adequate by
Proposition~\ref{propn:adeq-factors}, so it only remains to prove that $G$
is representably and corepresentably small. By duality, it is enough to
show that $G$ is representably small.

Let $c \in C$. Since $GF$ is small-codense, the dual of
Lemma~\ref{lemma:sd-colim} implies that $c = \lim_i GFa_i$ for some small
diagram $(a_i)$ in $\A$. Then
\[
\C(G-, c)
\iso
\lim_i \C(G-, GFa_i)
\iso
\lim_i \B(-, Fa_i)
\]
as $G$ is full and faithful. Hence $\C(G-, c)$ is a small limit in
$[\B^\op, \Set]$ of representables. 

Since $\B$ is gentle, the subcategory $\cocmp{\B}$ of $[\B^\op, \Set]$ is
complete. Limits in $\cocmp{\B}$ are computed pointwise because it
contains the representables (as noted by Day and Lack in Section~3
of~\cite{DaLa}). Hence $\cocmp{\B}$ is closed under small limits in
$[\B^\op, \Set]$, and in particular, $\C(G-, c)$ is small.
\end{proof}

\section{Characterization of the reflexive completion}
\label{sec:char}

Here we prove a theorem characterizing the reflexive completion of a
category uniquely up to equivalence. It is a refinement and variant of
Theorem~1.8 of Isbell~\cite{IsbeAS}. Roughly put, the result is that the
reflexive completion of a category $\A$ is the largest category into which
$\A$ embeds as a small-adequate subcategory. This is formally similar to
the fact that the completion of a metric space $A$ is the largest metric
space in which $A$ is dense.

\begin{lemma}
\label{lemma:dense-conj}
Let $F \from \A \to \B$ be a full and faithful small-dense
functor.  Then $\du{\B(F-, b)} \iso \B(b, F-)$ naturally in $b \in \B$;
that is, the diagram
\[
\xymatrix@R=2ex@C=3em{
&
\cocmp{\A}
\ar[dd]^{\vee}   \\
\B 
\ar[ru]^-{\nv{F}}
\ar[rd]_-{\conv{F}}
&
\\
&
[\A, \Set]^\op
}
\]
commutes up to a canonical natural isomorphism.
\end{lemma}

The smallness hypothesis guarantees that $\du{\B(F-, b)}$ is defined.

\begin{proof}
By the hypotheses on $F$,
\begin{align*}
\du{\B(F-, b)}(a)       &
\iso
\cocmp{\A}(\B(F-, b), \A(-, a))  \\
&
\iso
\cocmp{\A}(\B(F-, b), \B(F-, Fa))        \\
&
\iso
\B(b, Fa)
\end{align*}
naturally in $a \in \A$ and $b \in \B$.
\end{proof}

For a representably small functor $F \from \A \to \B$, the nerve functor
$\nv{F}$ has image in $\cocmp{\A}$. When does it have image in $\refl(\A)$?
The next result provides an answer (given without proof as statement~1.5
of~\cite{IsbeAS}).

\begin{propn}[Isbell]
\label{propn:codense-refl}
Let $F \from \A \to \B$ be a full and faithful small-dense functor. Then
$\B(F-, b)$ is reflexive for each $b \in \B$ if and only if $F$ is
small-adequate. 
\end{propn}

\begin{proof}
Suppose that $\B(F-, b)$ is reflexive for each $b \in \B$. Then we
have functors
\[
\A \toby{F} \B \toby{\nv{F}} \refl(\A)
\]
whose composite is the Yoneda embedding $J_\A \from \A \to \refl(\A)$
(Lemma~\ref{lemma:ff}). By Proposition~\ref{propn:incl-adeq}, $J_\A$ is
small-adequate. Hence by Proposition~\ref{propn:adeq-factors}, $F$ is
small-adequate.

Conversely, suppose that $F$ is small-adequate. Let $b \in \B$. Then by
Lemma~\ref{lemma:dense-conj} and its dual, there are canonical isomorphisms
\[
\ddu{\B(F-, b)}
\iso
\du{\B(b, F-)}
\iso
\B(F-, b),
\]
and $\B(F-, b)$ is reflexive.
\end{proof}

\begin{cor}
\label{cor:rnv}
Let $F \from \A \to \B$ be a small-adequate functor. Then there is a
functor $\rnv{F} \from \B \to \refl(\A)$, unique up to isomorphism, such
that the diagram
\[
\xymatrix@C+3em{
        &       &\cocmp{\cat{A}}        \\
\cat{B} 
\ar[rru]^{\nv{F}} 
\ar[r]|{\rnv{F}} 
\ar[rrd]_{\conv{F}}    &
\refl(\cat{A}) 
\ar@{^{(}->}[ru] 
\ar@{^{(}->}[rd]
        &       \\
        &       &\cmp{\cat{A}}
}
\]
commutes up to isomorphism. Moreover, $\rnv{F}$ is full and faithful.
\qed
\end{cor}

Precomposing this whole diagram with the functor $F \from \A \to \B$ gives
the diagram~\eqref{eq:inclusions} of Yoneda embeddings, by
Lemma~\ref{lemma:ff}. 

The main theorem of this section is as follows.

\begin{thm}
\label{thm:compl-char}
Let $F \from \A \to \B$ be a small-adequate functor. Then the functor
$\rnv{F} \from \B \to \refl(\A)$ is adequate, and the triangle
\[
\xymatrix{
\B 
\ar@{.>}[rr]^{\rnv{F}}    &       
&
\refl(\A)       \\
&
\A 
\ar[ul]^F 
\ar@{^{(}->}[ur]_{J_\A}        
}
\]
commutes up to canonical isomorphism. Moreover, up to isomorphism,
$\rnv{F}$ is the unique full and faithful functor $\B \to \refl(\A)$ such
that the triangle commutes.
\end{thm}

This result is mostly due to Isbell (Theorem~1.8 of~\cite{IsbeAS}). He
proved a version for properly adequate functors (Remark~\ref{rmk:proper}),
but without the conclusion that the functor $\B \to \refl(\A)$ is adequate or
unique.

\begin{proof}
The triangle commutes by Lemma~\ref{lemma:ff}, $\rnv{F}$ is full and
faithful by Corollary~\ref{cor:rnv}, and then $\rnv{F}$ is adequate by
Propositions~\ref{propn:incl-adeq} and~\ref{propn:adeq-factors}.
For uniqueness, the same two propositions prove the adequacy of any full and
faithful functor making the triangle commute, and the result follows from
Lemma~\ref{lemma:adeq-epi}.
\end{proof}

Theorem~\ref{thm:compl-char} characterizes the reflexive completion
uniquely up to equivalence. Indeed, given a category $\A$, form the
2-category whose objects are small-adequate functors out of $\A$ and whose
maps are adequate functors between their codomains making the evident
triangle commute.  Theorem~\ref{thm:compl-char} states that its terminal
object (in a 2-categorical sense) is the Yoneda embedding $J_\A\from \A
\incl \refl(\A)$.

\begin{remark}
\label{rmk:char-sub}
Corollary~\ref{cor:intermediate-adeq} and Theorem~\ref{thm:compl-char}
together imply that the categories containing $\A$ as a small-adequate
subcategory are, up to equivalence, precisely the full subcategories of
$\refl(\A)$ containing $\A$. When $\B$ is a
full subcategory of $\refl(\A)$ containing $\A$, writing $F \from \A \incl
\B$ for the inclusion, the uniqueness part of Theorem~\ref{thm:compl-char}
implies that $\rnv{F}$ is the inclusion $\B \incl \refl(\A)$.
\end{remark}

\begin{example}
\label{eg:char-DM}
Let $A$ be a poset, and recall Examples~\ref{eg:compl-posets}
and~\ref{eg:poset-adeq}. Loosely, Theorem~\ref{thm:compl-char} for $A$
states that its Dedekind--MacNeille completion is the largest ordered class
containing $A$ and with the property that every element can be expressed as
both a join and a meet of elements of $A$. For example, any poset
containing $\Q$ as a join- and meet-dense full subposet embeds into
$[-\infty, \infty]$.
\end{example}

\begin{example}
\label{eg:char-7}
The following example is due to Isbell (Example~1 of~\cite{IsbeSAS}). Let
$\B$ be the category of sets and partial bijections, and let $\A$ be the
full subcategory consisting of a single two-element set. Thus, $\A$
corresponds to a $7$-element monoid. Isbell showed that the inclusion $\A
\incl \B$ is adequate. It is small-adequate since $\A$ is small. Hence
by Theorem~\ref{thm:compl-char}, there is an adequate functor $\B \to
\refl(\A)$. In particular, there is a full and faithful functor from a
large category into $\refl(\A)$, so $\refl(\A)$ is large (in the strong
sense that it is not equivalent to any small category). This proves the
statement in Example~\ref{eg:compl-7}: the reflexive completion of a
finite category can be large.
\end{example}

Theorem~\ref{thm:compl-char} shows that when $F$ is small-adequate,
$\rnv{F}$ is adequate. But $\rnv{F}$ need not be \emph{small-}adequate, as
the following lemma and example show.

\begin{lemma}
\label{lemma:NF-sa}
Let $F \from \A \to \B$ be a small-adequate functor such that $\rnv{F}$ is
small-adequate. Then for every full and faithful functor $G \from \B \to
\C$ such that $GF$ is small-adequate, $G$ is also small-adequate.
\end{lemma}

For a general small-adequate $F$, without the hypothesis that $\rnv{F}$ is
small-adequate, Proposition~\ref{propn:adeq-factors} implies that any such
functor $G$ is adequate. So the force of the conclusion is that $G$ is
\emph{small}-adequate.

\begin{proof}
The proof will use the functors in the following diagram.
\[
\xymatrix@R-1em@C-1em{
\C
\ar[rrrr]^{\rnv{GF}}     &
        &       &       &
\refl(\A)       \\
        &
\B
\ar[lu]^G
\ar[rrru]|{\rnv{F}}     &
        &       &       \\
        &       &
\A
\ar[lu]^F
\ar[rruu]_{J_\A}        &
        &      
}
\]
Let $G$ be a full and faithful functor such that $GF$ is
small-adequate. Then there is an induced adequate functor $\rnv{GF}$ as
shown. Also, since $GF$ is adequate, Proposition~\ref{propn:adeq-factors}
implies that $G$ is adequate. Hence by Lemma~\ref{lemma:adeq-comp},
$\rnv{GF} \of G$ is adequate. Now by Theorem~\ref{thm:compl-char},
$\rnv{F}$ is the unique adequate functor satisfying $\rnv{F} \of F =
J_\A$. Since also $\rnv{GF} \of G \of F = J_\A$ by definition of
$\rnv{GF}$, we have $\rnv{GF} \of G = \rnv{F}$. The hypothesis that
$\rnv{F}$ is small-adequate and Proposition~\ref{propn:adeq-factors} then
imply that $G$ is small-adequate.
\end{proof}

\begin{example}
\label{eg:rnv-not-sa}
Let $F$ be the inclusion $A \incl B$ of Example~\ref{eg:ordinals}. As shown
there, the conclusion of Lemma~\ref{lemma:NF-sa} is false for $F$. Hence
$\rnv{F}$ is not small-adequate.
\end{example}

So, for full subcategories $\A \sub \B \sub \refl(\A)$, it is true that
$\A \incl \B$ is small-adequate and $\B \incl \refl(\A)$ is adequate,
but $\B \incl \refl(\A)$ need not be \emph{small}-adequate.

\section{Functoriality of the reflexive completion}
\label{sec:func}

The reflexive completion differs from many other completions in that it is
only functorial in a very restricted sense. First, there is no way to make
it act on \emph{all} functors:

\begin{propn}
\label{propn:func-fail}
There is no covariant or contravariant pseudofunctor $\cat{Q}$ from $\CAT$ to
$\CAT$ such that $\cat{Q}(\A) \eqv \refl(\A)$ for all $\A \in \CAT$.
\end{propn}

\begin{proof}
Suppose that there is.  Write $C_2$ for the two-element group, viewed as a
one-object category.  Then $C_2$ is a retract of $C_2 \times C_2$, so
$\cat{Q}(C_2)$ is a retract (up to natural isomorphism) of $\cat{Q}(C_2
\times C_2)$. Hence $\cat{Q}(C_2)$ has at most as many isomorphism classes of
objects as $\cat{Q}(C_2 \times C_2)$. But by Example~\ref{eg:compl-groups},
$\cat{Q}(C_2)$ has four isomorphism classes and $\cat{Q}(C_2 \times C_2)$ has
three, a contradiction.
\end{proof}

Second, reflexive completion is functorial in the following sense. For a
small-adequate functor $F \from \A \to \B$, the composite $J_\B \of F$ is
small-adequate by Lemma~\ref{lemma:adeq-comp}, giving a functor
\[
\refl(F) = N(J_\B \of F) \from \refl(\B) \to \refl(\A).
\]
By Theorem~\ref{thm:compl-char}, $\refl(F)$ is adequate, and up to
isomorphism, it is the unique full and faithful functor such that the diagram
\begin{align}
\label{eq:func}
\begin{array}{c}
\xymatrix{
\refl(\A)       &
\refl(\B)
\ar[l]_{\refl(F)}       \\
\A 
\ar[r]_F
\ar@{^{(}->}[u]^{J_\A}  &
\B
\ar@{^{(}->}[u]_{J_\B}
}
\end{array}
\end{align}
commutes. The uniqueness implies that $\refl$ defines a pseudofunctor
\begin{multline*}
\refl \from 
(\text{categories and small-adequate functors})^\op     \\
\to 
(\text{categories and adequate functors}).
\end{multline*}

When the reflexive completion of a category is regarded as a subcategory of
the $\Set$-valued functors on it, $\refl(F)$ is simply composition with
$F$:

\begin{lemma}
Let $F \from \A \to \B$ be a small-adequate functor. Then the squares
\[
\xymatrix{
[\A^\op, \Set]  &
[\B^\op, \Set]
\ar[l]_{- \of F}        \\
\refl(\A)
\ar@{^{(}->}[u] &
\refl(\B)
\ar@{^{(}->}[u]
\ar[l]^{\refl(F)}
}
\qquad
\xymatrix{
[\A, \Set]^\op  &
[\B, \Set]^\op
\ar[l]_{- \of F}        \\
\refl(\A)
\ar@{^{(}->}[u] &
\refl(\B)
\ar@{^{(}->}[u]
\ar[l]^{\refl(F)}
}
\]
commute up to canonical isomorphism.
\end{lemma}

\begin{proof}
Identify $\refl(\A)$ with the category of reflexive functors $\A^\op \to
\Set$, and similarly for $\B$. For $Z \in \refl(\B)$,
\[
(\refl(F))(Z) 
=
(\rnv{J_B \of F})(Z)
=
\refl(B)\bigl( (J_\B \of F)(-), Z \bigr),
\]
which at $a \in \A$ gives
\[
\bigl( (\refl(F))(Z) \bigr)(a)
=
[B^\op, \Set]\bigl( B(-, Fa), Z \bigr)
\iso
Z(Fa).
\]
Hence $\refl(F) \iso - \of F$, proving the commutativity of the first
square. The second follows by duality.
\end{proof}

For example, when $F$ is the inclusion of a small-adequate subcategory and
reflexive completions are viewed as categories of functors, $\refl(F)$ is
restriction. 

The pseudofunctor $\refl$ applied to a small-adequate functor $F$ produces
a functor $\refl(F)$ that is adequate but not always small-adequate: 

\begin{thm}
\label{thm:sa-eqv}
Let $F \from \A \to \B$ be a small-adequate functor. The following are
equivalent:
\begin{enumerate}
\item
\label{part:se-nsa}
$\rnv{F}$ is small-adequate;

\item 
\label{part:se-rsa}
$\refl(F)$ is small-adequate;

\item
\label{part:se-re}
$\refl(F)$ is an equivalence.
\end{enumerate}
If these conditions hold then $\rrnv{F}$ is defined and 
pseudo-inverse to $\refl(F)$.
\end{thm}

These equivalent conditions do not always hold, by
Example~\ref{eg:rnv-not-sa}.  

\begin{proof}
First, since $\refl(F) \of J_\B$ is an adequate functor satisfying
$\refl(F) \of J_\B \of F \iso J_\A$ (diagram~\eqref{eq:func}),
Theorem~\ref{thm:compl-char} gives
\begin{align}
\label{eq:rjn}
\refl(F) \of J_\B \iso \rnv{F},
\end{align}

Trivially, \bref{part:se-re} implies~\bref{part:se-rsa}. Now
assuming~\bref{part:se-rsa}, equation~\eqref{eq:rjn}
gives~\bref{part:se-nsa}, since the composite of small-adequate functors is
small-adequate.

Finally, assume~\bref{part:se-nsa}. Then $\rrnv{F}$ is defined, and we
show that it is pseudo-inverse to $\refl(F)$, proving~\eqref{part:se-re}
and the final assertion. Consider the diagram
\[
\xymatrix@R+1.5em@C+1.5em{
\refl(\A)
\ar@/_/[r]_{\rrnv{F}}   &
\refl(\B)
\ar@/_/[l]_{\refl(F)}   \\
\A 
\ar@{^{(}->}[u]^{J_\A}  
\ar[r]_F        &
\B.
\ar@{^{(}->}[u]_{J_\B}  
\ar[ul]|{\rnv{F}}
}
\]
The bottom-left triangle and the top-right triangle involving $\rrnv{F}$
commute by Theorem~\ref{thm:compl-char}, the top-right triangle
involving $\refl(F)$ commutes by equation~\eqref{eq:rjn}, and the outer
square commutes by definition of $\refl(F)$. Simple diagram
chases then show that
\[
\refl(F) \of \rrnv{F} \of J_\A \iso J_\A,
\quad
\rrnv{F} \of \refl(F) \of J_\B \iso J_\B.
\]
Hence by Lemma~\ref{lemma:adeq-epi}, $\refl(F)$ and $\rrnv{F}$ are
mutually pseudo-inverse. 
\end{proof}

It follows that reflexive completion is idempotent:

\begin{cor}[Isbell]
\label{cor:rc-idem}
For every category $\A$, the functors
\[
\xymatrix{
\refl(\A) 
\ar@/^/[r]^{J_{\refl(\A)}}       &
\refl\refl(\A)
\ar@/^/[l]^{\refl(J_\A)}
}
\]
define an equivalence $\refl(\A) \eqv \refl\refl(\A)$.
\end{cor}

A version of this result appeared as part of Theorem~1.8 of
Isbell~\cite{IsbeAS}, with a partial proof.

\begin{proof}
Take $F = J_\A \from \A \to \refl(\A)$ in
Theorem~\ref{thm:sa-eqv}. We have $\rnv{J_\A} = 1_{\refl(\A)}$, which
is certainly small-adequate, so $\refl(J_\A)$ and $\rrnv{J_\A}$ are
pseudo-inverse. But $\rrnv{J_\A} = \rnv{1_{\refl(\A)}} = J_{\refl(\A)}$.
\end{proof}

\begin{cor}
\label{cor:rc-tfae}
The following conditions on a category $\A$ are equivalent:
\begin{enumerate}
\item 
\label{part:rt-emb}
$J_\A \from \A \incl \refl(\A)$ is an equivalence;

\item
\label{part:rt-contra}
every reflexive functor $\A^\op \to \Set$ is representable;

\item
\label{part:rt-co}
every reflexive functor $\A \to \Set$ is representable;

\item
\label{part:rt-rc}
$\A \eqv \refl(\B)$ for some category $\B$.
\end{enumerate}
\end{cor}

\begin{proof}
$J_\A$ is always full and faithful, so it is an equivalence just when it is
essentially surjective on objects. Hence
\bref{part:rt-emb}$\iff$\bref{part:rt-contra}$\iff$\bref{part:rt-co} by
diagram~\eqref{eq:inclusions}. That~\bref{part:rt-rc} is equivalent
to~\bref{part:rt-emb} follows from Corollary~\ref{cor:rc-idem}.
\end{proof}

A category satisfying the equivalent conditions of
Corollary~\ref{cor:rc-tfae} is \demph{reflexively complete}. It is a
self-dual condition: $\A$ is reflexively complete if and only if $\A^\op$
is. 

In the introduction to Section~\ref{sec:char}, we compared reflexive
completion to metric completion, drawing an analogy between
Theorem~\ref{thm:compl-char} and the characterization of the completion of
a metric space $A$ as the largest metric space in which $A$ is
dense. The completion of a metric space $A$ can also be characterized as
the smallest complete metric space containing $A$. However, the reflexive
analogue of that characterization is false:

\begin{example}
Let $F \from \A \to \B$ be a small-adequate functor such that $\rnv{F}$ is
not small-adequate, as in Example~\ref{eg:rnv-not-sa}. Then by
Theorem~\ref{thm:sa-eqv}, the full and faithful functor $\refl(F)
\from \refl(\B) \to \refl(\A)$ is not an equivalence. Its image is a
full subcategory of $\refl(\A)$ that is reflexively complete and contains
$\A$, but is strictly smaller than $\refl(\A)$.
\end{example}

Such examples can be excluded by restricting to categories that are
gentle. There, the pseudofunctor $\refl$ acts somewhat trivially, in the
sense of part~\bref{part:gf-main} of Corollary~\ref{cor:gentle-func} below.

\begin{propn}
\label{propn:gentle-sa}
Let $F \from \A \to \B$ be a small-adequate functor. If $\B$ is gentle then
$\rnv{F}$ is small-adequate.
\end{propn}

\begin{proof}
Apply Lemma~\ref{lemma:gentle-comp} with $G = \rnv{F}$, recalling that
$\rnv{F} \of F$ is the small-adequate functor $J_\A$.
\end{proof}

\begin{cor}
\label{cor:gentle-func}
Let $\B$ be a gentle category. Then:
\begin{enumerate}
\item
\label{part:gf-main}
for every category $\A$ and small-adequate functor $F \from \A \to \B$, the
functor $\refl(F) \from \refl(\B) \to \refl(\A)$ is an equivalence;

\item
\label{part:gf-sub}
$\refl(\A) \eqv \refl(\B)$ for every small-adequate subcategory $\A \sub \B$.
\end{enumerate}
\end{cor}

\begin{proof}
Part~\bref{part:gf-main} follows from Proposition~\ref{propn:gentle-sa} and
Theorem~\ref{thm:sa-eqv}, and part~\bref{part:gf-sub} is then immediate.
\end{proof}

For example, every small-adequate subcategory of a complete and cocomplete
category $\B$ has the same reflexive completion as $\B$, which by
Proposition~\ref{propn:coco-rc} below is $\B$ itself.

\section{Reflexive completion and Cauchy completion}
\label{sec:cauchy}

Some formal resemblances are apparent between the reflexive and Cauchy
completions. Both are idempotent completions; both commute with the
operation of taking opposites; and to the analogy between reflexive
and metric completion mentioned in the introduction to
Section~\ref{sec:char}, one can add the fact that metric completion is
Cauchy completion in the $[0, \infty]$-enriched setting. On the other hand,
the reflexive and Cauchy completions are different, as even the example of
the empty category shows (Example~\ref{eg:compl-empty}). In this section,
we describe the relationship between them.

\begin{propn}
\label{propn:cc-basics}
Let $\A$ be a category.
\begin{enumerate}
\item
\label{part:cb-closed}
In $[\A^\op, \Set]$, the class of reflexive functors is closed under
small absolute colimits.

\item
\label{part:cb-cc}
$\refl(\A)$ is Cauchy complete.

\item
\label{part:cb-contains}
$\ovln{\A} \sub \refl(\A)$, when the Cauchy completion $\ovln{\A}$ and
reflexive completion $\refl(\A)$ are viewed as subcategories of $[\A^\op,
\Set]$.

\item
\label{part:cb-eqv}
$\refl(\ovln{\A}) \eqv \refl(\A)$.
\end{enumerate}
\end{propn}

\begin{proof}
For~\bref{part:cb-closed}, let $X = \colim_i X_i$ be a small absolute
colimit of reflexive functors $X_i$ in $[\A^\op, \Set]$. Each $X_i$ is
small, so $X$ is small, and $X$ is the absolute colimit of the $X_i$ in
$\cocmp{\A}$. Since $\dubk \from \cocmp{\A} \to [\A, \Set]^\op$ preserves
this colimit, $\du{X} \iso \colim_i X_i^\vee$ in $[\A, \Set]^\op$, again an
absolute colimit. Since each $X_i$ is reflexive, each $X_i^\vee$ is
small. Now $\cmp{\A} \sub [\A, \Set]^\op$ is complete, hence Cauchy
complete, hence closed under small absolute colimits in $[\A, \Set]^\op$.
So $\du{X}$ is small and is the absolute colimit of the $X_i^\vee$ in
$\cmp{\A}$. Since $\dubk \from \cmp{\A} \to [\A^\op, \Set]$ preserves this
colimit, $\ddu{X} \iso \colim_i X_i^{\vee\vee}$. Since each $X_i$ is
reflexive, so is $X$.

This proves~\bref{part:cb-closed}. Colimits in $\refl(\A) \sub
[\A^\op, \Set]$ are computed pointwise, so $\refl(\A)$ has absolute colimits,
proving~\bref{part:cb-cc}. And $\ovln{\A}$ is the closure of $\A \sub
[\A^\op, \Set]$ under absolute colimits, giving~\bref{part:cb-contains}.

For~\bref{part:cb-eqv}, Corollary~\ref{cor:intermediate-adeq}
and~\bref{part:cb-contains} imply that the inclusion $F \from \A \incl
\ovln{\A}$ is small-adequate. We will prove that the induced adequate
functor $\rnv{F} \from \ovln{\A} \to \refl(\A)$ is representably small. It
will follow by duality that $\rnv{F}$ is small-adequate, and so by
Theorem~\ref{thm:sa-eqv} that $\refl(F)\from \refl(\ovln{\A}) \to
\refl(\A)$ is an equivalence. 

First recall that by the 2-universal property of Cauchy completion,
restriction along $F$ is an equivalence
\begin{align}
\label{eq:cc-restr}
[\ovln{\A}^\op, \Set] \toby{\eqv} [\A^\op, \Set].
\end{align}
Its pseudo-inverse is left Kan extension along $F$, and left Kan extension
along any functor preserves smallness, so every $Z \from \ovln{\A}^\op \to
\Set$ such that $Z|_{\A^\op}$ is small is itself small.

To prove that $\rnv{F} \from \ovln{\A} \to \refl(\A)$ is representably
small, we regard $\ovln{\A}$ and $\refl(\A)$ as subcategories of $[\A^\op,
\Set]$ and recall that $\rnv{F}$ is then the inclusion $\ovln{\A} \incl
\refl(\A)$ (Remark~\ref{rmk:char-sub}). For each $X \in \refl(\A)$, the
Yoneda lemma gives
\[
\refl(\A)(-, X)|_{\A^\op} \iso X,
\]
so $\refl(\A)(-, X)|_{\A^\op}$ is small, so
$\refl(\A)(-, X)|_{\ovln{\A}^\op}$ is small, as required.
\end{proof}

\begin{remark}
There is another, more elementary, proof of~\bref{part:cb-eqv}. One shows
that the equivalence~\eqref{eq:cc-restr} restricts to an equivalence
$\cocmp{\ovln{\A}} \toby{\eqv} \cocmp{\A}$, and dually. Then one shows that
the conjugacy operations on $\ovln{\A}$ and $\A$ commute with these
equivalences. It follows that the equivalence~\eqref{eq:cc-restr} also
restricts to an equivalence $\refl(\ovln{\A}) \toby{\eqv} \refl(\A)$.
\end{remark}

Proposition~\ref{propn:cc-basics}\bref{part:cb-eqv} implies:

\begin{cor}
\label{cor:morita}
Morita equivalent categories have equivalent reflexive completions.
\qed
\end{cor}

Hence the category $\refl(\A)$ is determined by the category $[\A^\op,
\Set]$, without knowledge of $\A$ itself.

Conjugacy (as opposed to reflexivity) also plays a role in the theory
of Cauchy completion. For vector spaces $X$ and $Z$, there is a canonical
linear map 
\begin{align}
\label{eq:lin}
\begin{array}{ccc}
\du{X} \otimes Z        &\to            &\Vect(X , Z)    \\
\xi \otimes z           &\mapsto        &\xi(-)\cdot z,
\end{array}
\end{align}
where $\du{X}$ is the linear dual of $X$. Analogously, for a category $\A$
and $X, Z \in \cocmp{\A}$, there is a canonical map of sets
\[
\kappa_{X, Z} \from \du{X} \pof Z \to \cocmp{\A}(X, Z),
\]
to be defined. Here 
\[
\du{X} \pof Z = \int^a \du{X}(a) \times Z(a),
\]
and the coend exists since $Z$ is small. The map $\kappa_{X, Z}$ can be
defined concretely by specifying a natural family of functions
\[
\du{X}(a) \times Z(a) \to [X(b), Z(b)]
\]
($a, b \in \A$), which are taken to be
\[
(\xi, z) \mapsto \Bigl( x \mapsto \bigl(Z(\xi_b(x))\bigr)(z) \Bigr).
\]
Equivalently, $\du{X} \pof Z$ is the composite profunctor
\[
\xymatrix{
\One 
\ar[r]|-@{|}^Z  &
\A 
\ar[r]|-@{|}^{\du{X}}   &
\One,
}
\]
and the map $\kappa_{X, Z}$ corresponds under the
adjunctions~\eqref{eq:prof-adj} to 
\[
X \pof \du{X} \pof Z 
\ltoby{\epsln_X \pof Z} 
\Hom_\A \pof Z \iso Z,
\]
where $\epsln_X \from X \pof \du{X} \to \Hom_\A$ is the natural
transformation of~\eqref{eq:epsln}.

\begin{propn}
\label{propn:cc-conj}
Let $\As$ be a small category. The following conditions on a functor $X
\from \As^\op \to \Set$ are equivalent:
\begin{enumerate}
\item 
\label{part:ccc-cc}
$X \in \ovln{\A}$, when $\ovln{\A}$ is regarded as a subcategory of
$\cocmp{\As}$;

\item
\label{part:ccc-sp}
$\cocmp{\A}(X, -) \from \cocmp{\A} \to \Set$ preserves small colimits;

\item
\label{part:ccc-ra}
$X \from \One \pfun \A$ has a right adjoint in the bicategory $\Prof$;

\item
\label{part:ccc-rac}
$X \from \One \pfun \A$ has right adjoint $\du{X}$ in $\Prof$, with counit
$\epsln_X$;

\item
\label{part:ccc-hom}
$\kappa_{X, Z}\from \du{X} \pof Z \to \cocmp{\A}(X, Z)$ is a bijection
for all $Z \in \cocmp{\A}$.
\end{enumerate}
\end{propn}

\begin{proof}
The equivalence of~\bref{part:ccc-cc}--\bref{part:ccc-ra} is standard, and 
\bref{part:ccc-rac}$\implies$\bref{part:ccc-ra} is trivial. 

To prove \bref{part:ccc-ra}$\implies$\bref{part:ccc-rac}, suppose that $X$
has a right adjoint $Y$ in $\Prof$, with counit $\beta$. As in any
bicategory, this implies that $\beta$ exhibits $Y$ as the right Kan lift of
$\Hom_\A$ through $X$. (Here we are using a result
more often given in the dual form, involving Kan extensions. See, for
instance, Theorem~X.7.2 of Mac~Lane~\cite{MacLCWM}, which is stated for
$\CAT$ but the proof is valid in any bicategory.) But as shown at the end of
Section~\ref{sec:conj-small}, $\epsln_X$ exhibits $\du{X}$ as the right Kan
lift of $\Hom_\A$ through $X$. Hence $(\du{X}, \epsln_X) \iso (Y, \beta)$
and~\bref{part:ccc-rac} follows.

Now assume~\bref{part:ccc-hom}. The maps $\kappa_{X, Z}$ are natural in
$Z \in \cocmp{\A}$, so 
\[
(\du{X} \pof -) 
\iso 
\cocmp{\A}(X, -)
\from
\cocmp{\A} \to \Set.
\]
But $\du{X} \pof -$ preserves small colimits by the adjointness
relations~\eqref{eq:prof-adj}, giving~\bref{part:ccc-sp}.

Finally, assuming~\bref{part:ccc-sp}, we prove~\bref{part:ccc-hom}. When
$Z$ is representable, the function
\[
\kappa_{X, Z} \from \du{X} \pof Z \to \cocmp{\A}(X, Z) 
\]
is bijective. But every $Z \in \cocmp{\A}$ is a small colimit of
representables, and both $\du{X} \pof -$ and $\cocmp{\A}(X, -)$ preserve
small colimits, so $\kappa_{X, Z}$ is bijective for all $Z$.
\end{proof}

This result can be generalized to enriched categories. When $\A$ is a
one-object $\Ab$-category corresponding to a field, $\ovln{\A}$ is the
category of finite-dimensional vector spaces, the maps $\kappa_{X, Z}$ are
as defined in equation~\eqref{eq:lin}, and we recover the following fact: a
vector space $X$ is finite-dimensional if and only if $\kappa_{X, Z}$ is an
isomorphism for all vector spaces $Z$.

Further results on Isbell conjugacy and Cauchy completion can be found in
Sections~6 and~7 of Kelly and Schmitt~\cite{KeSc}.

\section{Limits in reflexive completions}
\label{sec:limits}

A partially ordered set is complete if and only if it is reflexively
complete (Example~\ref{eg:compl-posets}). In one direction, we show that,
more generally, every complete or cocomplete category is reflexively
complete (Proposition~\ref{propn:coco-rc}). But the converse is another
matter entirely: in general, reflexively complete categories have very few
limits. We identify exactly which ones.

\begin{lemma}
\label{lemma:incl-lims}
Let $\A$ be a category.
\begin{enumerate}
\item 
\label{part:first-lims}
The inclusion $J_\A \from \A \incl \refl(\A)$ 
preserves and reflects both limits and colimits.

\item
\label{part:second-lims}
The inclusion $\refl(\A) \incl [\A^\op, \Set]$ preserves limits and reflects
both limits and colimits.
\end{enumerate}
\end{lemma}

The inclusion $\refl(\A) \incl [\A^\op, \Set]$ does not preserve
\emph{co}limits, since $J_\A \from \A \incl \refl(\A)$ preserves colimits
but the Yoneda embedding $\A \incl [\A^\op, \Set]$ does not.

\begin{proof}
The statements on reflection are immediate, since both functors are full
and faithful.  

For~\bref{part:first-lims}, the embedding $J_\A \from \A \incl \refl(\A)$ is
adequate by Proposition~\ref{propn:incl-adeq}, so preserves limits and
colimits by Lemma~\ref{lemma:dense-cts} and its dual.

For~\bref{part:second-lims}, the composite of $\refl(\A) \incl [\A^\op,
\Set]$ with $J_\A \from \A \incl \refl(\A)$ is the Yoneda embedding, which
is dense, so $\refl(\A) \incl [\A^\op, \Set]$ is dense by
Lemma~\ref{lemma:dense-factors}. Since it is also full and faithful, it
preserves limits by Lemma~\ref{lemma:dense-cts}.
\end{proof}

We have already shown that reflexively complete categories are Cauchy
complete, that is, have absolute limits and colimits
(Proposition~\ref{propn:cc-basics}\bref{part:cb-cc}). The next few results
show that the reflexive completion of a \emph{small} category also has
initial and terminal objects, but that the (co)limits just mentioned are
the only ones that generally exist.

For a small category $\A$ and $b \in \A$, write $\Cone(\id, b)$ for the set
of cones from the identity to $b$ (natural transformations from $\id_\A$ to
the constant endofunctor $b$). This defines a functor $\Cone(\id, -) \from
\A \to \Set$.

\begin{lemma}
\label{lemma:cone}
Let $\A$ be a small category. Then $\du{\Cone(\id, -)}$ is the terminal
functor $\A^\op \to \Set$. 
\end{lemma}

\begin{proof}
Fixing $a \in \A$, we must show that there is exactly one natural
transformation $\alpha \from \Cone(\id, -) \to \A(a, -)$.
There is at least one, since we can define $\alpha_b(p) = p_a$ for each $b
\in \A$ and $p \in \Cone(\id, b)$.

To prove uniqueness, let $\beta \from \Cone(\id, -) \to \A(a, -)$. Let $b
\in \B$ and $p \in \Cone(\id, b)$. We must prove that $\beta_b(p) =
p_a$. 

Since $p$ is a cone, $p_b \of p_c = p_c$ for all $c \in \A$. So we
have an equality of cones $p_b \of p = p$, and then naturality of $\beta$
gives $p_b \of \beta_b(p) = \beta_b(p)$. On the other hand, $p_b \of
\beta_b(p) = p_a$ since $p$ is a cone. Hence $\beta_b(p) = p_a$, as required.
\end{proof}

\begin{propn}
\label{propn:init-term}
The reflexive completion of a small category has initial and terminal
objects. 
\end{propn}

\begin{proof}
Let $\A$ be a small category. Write $1$ for the terminal functor
$\A^\op \to \Set$. Then $\du{1} \iso \Cone(\id, -)$, so $\ddu{1} \iso 1$ by
Lemma~\ref{lemma:cone}. Hence $1$ is reflexive. It follows that $1$ is a
terminal object of $\refl(\A) \sub [\A^\op, \Set]$. By duality
(Remark~\ref{rmk:refl-dual}), $\refl(\A)$ also has an initial object.
\end{proof}

\begin{remark}
When $\refl(\A)$ is viewed as a subcategory of $[\A^\op, \Set]$, its
initial object is not in general the initial (empty) functor $\A^\op \to
\Set$. The case $\A = \One$ already shows this
(Example~\ref{eg:compl-one}).
\end{remark}

Proposition~\ref{propn:init-term} is false for large
categories. The proof fails because the terminal functor need not be
small. Any large discrete category $\A$ is a counterexample to the
statement, since then $\refl(\A) \eqv \A$
(Example~\ref{eg:compl-large-discrete}).

To show that reflexive completions have no other (co)limits in general, we
use the following lemma. A category $\cat{I}$ is an \demph{absolute limit
shape} if all $\cat{I}$-limits are absolute.

\begin{lemma}
\label{lemma:abs-lim-cone}
A small category is an absolute limit shape if and only if it admits a cone
on the identity.
\end{lemma}

This result is at least implicit in the literature on Cauchy completeness,
but since we have been unable to find it stated explicitly, we include a
proof.

\begin{proof}
Let $\scat{I}$ be a small category admitting a cone
\[
\bigl( k \toby{u_i} i \bigr)_{i \in \scat{I}}
\]
on the identity. Take functors $\scat{I} \toby{D} \A \toby{F} \B$ and a
limit cone
\begin{align}
\label{eq:alc-lim}
\bigl( L \toby{p_i} Di \bigr)_{i \in \scat{I}}
\end{align}
on $D$. The cone $(Dk \toby{Du_i} Di)$ induces a map $Dk \toby{f} L$ in
$\A$. Then $f \of p_k = 1_L$. To show that $(FL \toby{Fp_i} FDi)$ is
a limit cone on $FD$, let
\[
\bigl( B \toby{r_i} FDi \bigr)_{i \in \scat{I}}
\]
be any other cone on $FD$. We have a map $B \toby{r_k} FDk \toby{Ff} FL$,
and it is routine to check that this is the unique map of cones from
$(r_i)$ to $(Fp_i)$. Hence the limit cone~\eqref{eq:alc-lim} is absolute.

We sketch the proof of the converse, which will not be needed here.
Suppose that $\scat{I}$-limits are absolute. The limit of the Yoneda
embedding $\scat{I} \incl \cocmp{\scat{I}}$ is $\Cone(-, \id_{\scat{I}})$. By
absoluteness, this limit is preserved by $\colim_\scat{I} \from
\cocmp{\scat{I}} \to \Set$. It follows that $\colim_{\scat{I}} (\Cone(-,
\id_{\scat{I}})) = 1$, so there exists a cone on $\id_{\scat{I}}$.
\end{proof}

\begin{thm}
\label{thm:rcn-lim}
Let $\scat{I}$ be a small category. The following are equivalent:
\begin{enumerate}
\item 
\label{part:rnl-rc}
$\scat{I}$-limits exist in the reflexive completion of every small category;

\item
\label{part:rnl-cc}
$\scat{I}$-limits exist in every Cauchy complete category with a terminal
object; 

\item
\label{part:rnl-exp}
$\scat{I}$ is empty or an absolute limit shape.
\end{enumerate}
\end{thm}

By Remark~\ref{rmk:refl-dual}, dual results hold for colimits.

\begin{proof}
Every Cauchy complete category has small absolute limits,
so~\bref{part:rnl-exp}$\implies$\bref{part:rnl-cc}. Every reflexive
completion of a small category is Cauchy complete with a terminal object
(Propositions~\ref{propn:cc-basics}\bref{part:cb-cc}
and~\ref{propn:init-term}), so
\bref{part:rnl-cc}$\implies$\bref{part:rnl-rc}. It remains to prove
\bref{part:rnl-rc}$\implies$\bref{part:rnl-exp}, which we do by
contradiction.

Assume~\bref{part:rnl-rc}, and that $\scat{I}$ is neither empty nor an
absolute limit shape. By Lemma~\ref{lemma:abs-lim-cone}, $\scat{I}$ admits
no cone on the identity.

Let $\scat{J}$ be the category obtained from $\scat{I}$ by adjoining a new
object $z$ and maps $p_i^0, p_i^1 \from z \to i$ for each $i \in \scat{I}$,
subject to $u \of p_i^\epsln = p_j^\epsln$ for each map $i \toby{u} j$ in
$\scat{I}$ and $\epsln \in \{0, 1\}$. By assumption, the composite
\[
\scat{I} \incl \scat{J} \incl \refl(\scat{J})
\]
has a limit, $L$. Since the inclusion $\refl(\scat{J}) \incl
[\scat{J}^\op, \Set]$ preserves limits (Lemma~\ref{lemma:incl-lims}), $L$ is
the limit of the composite 
\[
\scat{I} \incl \scat{J} \incl [\scat{J}^\op, \Set],
\]
and is reflexive. Now for $i \in \scat{I}$, 
\[
L(i) 
= 
\lim_{i' \in \scat{I}} \scat{J}(i, i') 
= 
\Cone(i, \id_{\scat{I}}) 
=
\emptyset, 
\]
and 
\[
L(z) 
= 
\lim_{i' \in \scat{I}} \scat{J}(z, i') 
\iso 
\{0, 1\}^{\pi_0 \scat{I}},
\]
where $\pi_0 \scat{I}$ is the set of connected-components of
$\scat{I}$. Write $S = L(z)$. Since $\scat{I}$ is nonempty, $|S| \geq
2$. Then
\[
L \iso S \times \scat{J}(-, z),
\]
so 
\[
\du{L} \iso \scat{J}(z, -)^S.
\]
Since $L$ is reflexive, the unit map $\eta_{L, z} \from L(z) \to
\ddu{L}(z)$ is surjective. That is, every natural transformation 
\begin{align}
\label{eq:rcnl-transf}
\alpha \from \scat{J}(z, -)^S \to \scat{J}(z, -)
\end{align}
is the $s$-projection for some $s \in S$. 

We will derive a contradiction by constructing a
transformation~\eqref{eq:rcnl-transf} not of this form. For $i \in
\scat{I}$, let $\alpha_i$ be the function
\[
\begin{array}{cccc}
\alpha_i\from   &\scat{J}(z, i)^S       &\to    &\scat{J}(z, i) \\[1ex]
        &
\bigl( p_i^{\epsln_s} \bigr)_{s \in S}  &
\mapsto &
p_i^{\min_s \epsln_s},
\end{array}
\]
and let $\alpha_z$ be the unique function $\scat{J}(z, z)^S \to
\scat{J}(z, z)$. It is routine to check that $\alpha$ defines a natural
transformation~\eqref{eq:rcnl-transf}. Hence $\alpha$ is $t$-projection for
some $t \in S$. 

Choose some $i \in \scat{I}$, as we may since $\scat{I}$ is
nonempty. Writing $\delta$ for the Kronecker delta and recalling that
$|S| \geq 2$,
\[
\alpha_i \bigl( \bigl(p_i^{\delta_{st}}\bigr)_{s \in S} \bigr)
=
p_i^{\min_s \delta_{st}}
=
p_i^0.
\]
But the $t$-projection of $\bigl(p_i^{\delta_{st}}\bigr)_{s \in S}$ is
$p_i^{\delta_{tt}} = p_i^1$, a contradiction.
\end{proof}

\begin{remark}
Theorem~\ref{thm:rcn-lim} might suggest the idea that reflexive completion is
Cauchy completion followed by the adjoining of initial and terminal
objects, and there are examples in Section~\ref{sec:eg-compl} where
this is indeed the case. But the group of order $2$
(Example~\ref{eg:compl-groups}) shows that this is false in general.
\end{remark}

Theorem~\ref{thm:rcn-lim} concerns limits in reflexive completions of small
categories. Every reflexively complete category is the reflexive completion
of some category (Corollary~\ref{cor:rc-tfae}), but not always of a
\emph{small} category. For example, any large discrete category is
reflexively complete (Example~\ref{eg:compl-large-discrete}), but does not
have a terminal object and so cannot be the reflexive completion of a small
category. The following corollary is the analogue of
Theorem~\ref{thm:rcn-lim} for the larger class of reflexively complete
categories.

\begin{cor}
\label{cor:rce-lim}
Let $\scat{I}$ be a small category. The following are equivalent:
\begin{enumerate}
\item 
\label{part:rel-rc}
$\scat{I}$-limits exist in every reflexively complete category;

\item
\label{part:rel-cc}
$\scat{I}$-limits exist in every Cauchy complete category; 

\item
\label{part:rel-exp}
$\scat{I}$ is an absolute limit shape.
\end{enumerate}
\end{cor}

\begin{proof}
Certainly~\bref{part:rel-exp}$\implies$\bref{part:rel-cc},
and~\bref{part:rel-cc}$\implies$\bref{part:rel-rc} because reflexively
complete categories are Cauchy complete.  Assuming~\bref{part:rel-rc},
$\scat{I}$ is empty or an absolute limit shape by
Theorem~\ref{thm:rcn-lim}. But any large discrete category is a reflexively
complete category with no terminal object, so $\scat{I}$ is not empty,
proving~\bref{part:rel-exp}.
\end{proof}

Every $\Set$-valued functor on a small category can be expressed as a small
colimit of representables. Not every such functor can be expressed as a
small \emph{limit} of representables, since then it would preserve small
limits.

\begin{lemma}
\label{lemma:lims-reps}
Let $\A$ be a category.  A functor $\A^\op \to \Set$ is a small limit of
representables if and only if it is the conjugate of some small functor $\A
\to \Set$.  
\end{lemma}

\begin{proof}
Let $X \from \A^\op \to \Set$. If $X \iso \lim_i \A(-, a_i)$ for some small
diagram $(a_i)$ in $\A$ then $X$ is the conjugate of the small functor
$\colim_i A(a_i, -)$. Conversely, every small functor $\A \to \Set$ can be
expressed as a small colimit of representables, and its conjugate is the
corresponding small limit of representables.
\end{proof}

\begin{propn}
\label{propn:refl-pres-lims}
Every reflexive $\Set$-valued functor preserves small limits.
\end{propn}

\begin{proof}
By Lemma~\ref{lemma:lims-reps}, every reflexive functor $X$ is a small
limit of representables. But representables preserve small limits, so $X$
does too. 
\end{proof}

A reflexive functor need not preserve \emph{co}limits. For example, the
unique reflexive functor $\One \to \Set$ is the terminal functor, which
does not preserve initial objects. 

Moreover, although a reflexive functor $\A \to \Set$ preserves all limits
that exist in $\A$, it need not be flat. Indeed, many of the examples in
Section~\ref{sec:eg-compl} are of categories $\A$ where the initial
functor $0 \from \A \to \Set$ is reflexive; but $0$ is not flat.

\begin{propn}
\label{propn:coco-rc}
Every complete or cocomplete category is reflexively complete.
\end{propn}

\begin{proof}
Let $\A$ be a complete category. By Lemma~\ref{lemma:lims-reps}, every
reflexive functor $X \from \A^\op \to \Set$ is a small limit of
representables. But $\A$ is complete, so $X$ is representable. This proves
that every complete category if reflexively complete. Since reflexive
completeness is a self-dual condition, the dual result follows.
\end{proof}

An alternative proof can be given, based on 
the last paragraph of Remark~\ref{rmk:gentle}.

\section{The Isbell envelope}
\label{sec:envelope}

This short section describes the relationship between two constructions due
to Isbell. As well as introducing the reflexive completion of a category
$\A$ in 1960~\cite{IsbeAS}, he also defined what
is now called the Isbell envelope $\IEnv(\A)$ in 1966 (naming it the
`couple category': \cite{IsbeSC}, p.~622). See also Garner~\cite{GarnIM}
for a thorough modern treatment of the Isbell envelope.

We begin with an arbitrary adjunction 
\[
\xymatrix@C+1em{
\cat{C}
\ar@/^/[r]^F_\bot     &
\cat{D},
\ar@/^/[l]^G
}
\]
with unit $\eta$ and counit $\epsln$. We already defined the invariant part
$\Inv(F, G)$ (Section~\ref{sec:completion}), which comes with full and
faithful inclusion functors
\[
\xymatrix{
\Inv(F, G)
\ar[r]
\ar[d] &
\cat{C} \\
\cat{D}
}
\]
The \demph{envelope} $\Env(F, G)$ of the adjunction is the category of
quadruples 
\[
(c \in \cat{C},\ d \in \cat{D},\ 
f \from c \to Gd,\ g \from Fc \to d)
\]
such that $f$ and $g$ are each other's transposes. A map $(c, d, f, g) \to
(c', d', f', g')$ in $\Env(F, G)$ is a pair of maps 
\[
(p \from c \to c',\ q \from d \to d')
\]
such that $(Gq)\of f = f' \of p$, or equivalently, $q \of g = g' \of
(Fp)$. There are canonical functors
\begin{align}
\label{eq:emb-Env}
\begin{array}{c}
\xymatrix{
        &
\cat{C}
\ar[d]  \\
\cat{D}
\ar[r]  &
\Env(F, G)
}
\end{array}
\end{align}
defined by
\[
c \mapsto (c, Fc, \eta_c, 1_{Fc}),
\qquad
d \mapsto (Gd, d, 1_{Gd}, \epsln_d)
\]
($c \in \cat{C}, d \in \cat{D}$), which are full and faithful. 

The invariant part and the envelope are related as follows.

\begin{lemma}
\label{lemma:Inv-Env}
Let $F \ladj G \from \cat{D} \to \cat{C}$ be an adjunction.
\begin{enumerate}
\item
The full and faithful functors
\[
\xymatrix{
\Inv(F, G)
\ar[r]
\ar[d]  &
\cat{C}
\ar[d]  \\
\cat{D}
\ar[r]  &
\Env(F, G)
}
\]
defined above form a 2-pullback square (in the up-to-isomorphism
sense). 

\item
The composite functor $\Inv(F, G) \to \Env(F, G)$ defines
an equivalence between $\Inv(F, G)$ and the full subcategory of objects
$(c, d, f, g)$ of $\Env(F, G)$ such that $f$ and $g$ are isomorphisms.
\end{enumerate}
\end{lemma}

\begin{proof}
The proof is a series of elementary checks, omitted here. In outline, an
object of the 2-pullback of the two functors into $\Env(F, G)$ consists of
objects $c \in \cat{C}$ and $d \in \cat{D}$ together with an isomorphism
\[
(c, Fc, \eta_c, 1_{Fc}) \toby{\iso} (Gd, d, 1_{Gd}, \epsln_d)
\]
in $\Env(F, G)$, and one verifies that this amounts to an object of
$\Inv(F, G)$.
\end{proof}

\begin{remark}
The $\Inv$ and $\Env$ constructions can be understood as adjoints.  Let
$\ADJ$ be the 2-category defined as follows. Objects are adjunctions
$(\cat{C}, \cat{D}, F, G)$. A map $(\cat{C}, \cat{D}, F, G) \to (\cat{C}',
\cat{D}', F', G')$ consists of functors and natural transformations
\[
\bigl(
\cat{C} \toby{P} \cat{C}',\ 
\cat{D} \toby{Q}, \cat{D}',\
F'P \toby{\alpha} QF,\
PG \toby{\beta} G'Q
\bigr)
\]
such that $\alpha$ and $\beta$ are each other's mates. The 2-cells are the
evident ones. Let $\ADJi$ be the sub-2-category of $\ADJ$ consisting
of all objects, just those maps $(P, Q, \alpha, \beta)$ for which $\alpha$
and $\beta$ are isomorphisms, and all 2-cells between them.

There are 2-functors and 2-adjunctions
\[
\xymatrix@C+2em@R-1.5em@M+.2em{
        &
\ADJ
\\
\CAT
\ar@{<-}@/^1pc/[ru]^-[@]{\Env}_-[@]{\top}      
\ar@{^{(}->}[ru]
\ar@{<-}@/_1pc/[rd]_-[@]{\Inv}^-[@]{\bot}      
\ar@{^{(}->}[rd]        &
        \\
        &
\ADJi
\ar@{^{(}->}[uu]
}
\]
Here, the embedding of $\CAT$ into $\ADJi$ associates to a category the
identity adjunction on it, and $\Inv$ is its right 2-adjoint. Similarly,
$\Env$ is the right 2-adjoint of $\CAT \incl \ADJ$.
\end{remark}

Now consider the conjugacy adjunction $\cocmp{\A} \oppairu \cmp{\A}$
of a small category $\A$. Its invariant part is $\refl(\A)$. Its
envelope is the \demph{Isbell envelope} $\IEnv(\A)$. An object of
$\IEnv(\A)$ is a quadruple
\[
(X \from \A^\op \to \Set,\ 
Y \from \A \to \Set,\ 
\phi \from X \to \du{Y},\
\psi \from Y \to \du{X})
\]
such that $\phi$ and $\psi$ correspond to one another under the conjugacy
adjunction. By the isomorphisms~\eqref{eq:etimes-adjn}, $\IEnv(\A)$ is
equivalently the category of triples
\[
(X \from \A^\op \to \Set,\ 
Y \from \A \to \Set,\ 
\chi \from X \etimes Y \to \Hom_\A),
\]
with the obvious maps between them (as Isbell observed in Section~1.1
of~\cite{IsbeSC}). With this formulation of $\IEnv(\A)$, the canonical
functor $\cocmp{\A} \to \IEnv(\A)$ (as in diagram~\eqref{eq:emb-Env}) maps
$X \in \cocmp{\A}$ to $(X, \du{X}, \epsln_X)$, where $\epsln_X$ is the
natural transformation of~\eqref{eq:epsln}. A dual statement holds for
$\cmp{\A}$.

Lemma~\ref{lemma:Inv-Env} immediately implies:

\begin{propn}
For a small category $\A$, the canonical full and faithful functors
\[
\xymatrix{
\refl(\A)      
\ar[r]
\ar[d]  &
\cocmp{\A}      
\ar[d]  \\
\cmp{\A}
\ar[r]  &
\IEnv(\A)
}
\]
form a 2-pullback square.
\qed
\end{propn}

Thus, informally, $\refl(\A) = \cocmp{\A} \cap \cmp{\A}$.

\paragraph{Acknowledgements} We thank the referee for their helpful and
insightful comments, and especially for the last paragraph of
Remark~\ref{rmk:gentle}.

\bibliography{mathrefs}
\end{document}